\input ./style/arxiv-general.cfg
\documentclass[aap,MSNbibl,seceqn,dvips]{arximspdf}
\makeatletter
   \@ifpackageloaded{graphicx}{}{\usepackage{graphicx}}
\makeatother

\doi{10.1214/14-AAP1048} 
\volume{25}
\issue{4}
\pubyear{2015}
\firstpage{2263}
\lastpage{2300}
\docsubty{FLA}

\makeatletter
\newcommand{\Id}{\operatorname{Id}}
\newcommand{\rrVert}{\Vert}
\newcommand{\llVert}{\Vert}
\newcommand{\rrvert}{\vert}
\newcommand{\llvert}{\vert}
\newcommand{\acc}{\operatorname{acc}}
\newcommand{\cG}{{\mathcal G}}
\newcommand{\tN}{{\mathbb N}}
\newcommand{\R}{\mathbb{R}}
\newcommand{\Z}{\mathbb{Z}}
\newcommand{\E}{\mathbb{E}}
\newtheorem{theorem}{Theorem}
\newtheorem{lemma}{Lemma}
\newtheorem{proposition}{Proposition}
\newproclaim{definition}{Definition}
\newproclaim{remark}{Remark}
\makeatother

\begin{document}
\begin{frontmatter}

\title{Optimal scaling for the transient phase of the random walk
Metropolis algorithm: The~mean-field~limit\thanksref{T1}}
\runtitle{Optimal scaling for the transient phase of RWM}

\begin{aug}
\author[A]{\fnms{Benjamin}~\snm{Jourdain}\ead[label=e1]{jourdain@cermics.enpc.fr}}, 
\author[A]{\fnms{Tony}~\snm{Leli\`evre}\ead[label=e2]{lelievre@cermics.enpc.fr}} 
\and
\author[B]{\fnms{B\l{}a\.{z}ej}~\snm{Miasojedow}\corref{}\ead[label=e3]{bmia@mimuw.edu.pl}} 
\runauthor{B. Jourdain, T. Leli\`evre and B. Miasojedow}
\affiliation{Universit\'e Paris-Est and INRIA,
Universit\'e Paris-Est and INRIA,\\
and
Universit\'e Paris-Est and University of Warsaw}
\address[A]{B. Jourdain\\
T. Leli\`evre\\
Universit\'e Paris-Est\\
Cermics (ENPC)\\
INRIA\\
F-77455 Marne-la-Vall\'ee\\
France\\
\printead{e1}\\
\phantom{E-mail: }\printead*{e2}}
\address[B]{B. Miasojedow\\
Universit\'e Paris-Est\\
Cermics (ENPC)\\
F-77455 Marne-la-Vall\'ee\\
France\\
and\\
Institute of Applied Mathematics\\
University of Warsaw\\
ul. Banacha 2\\
02-097 Warszawa\\
Poland\\
\printead{e3}} 
\end{aug}
\thankstext{T1}{Supported by the French National Research Agency under
the Grants ANR-08-BLAN-0218 (BigMC) and ANR-09-BLAN-0216-01 (MEGAS).}

\received{\smonth{10} \syear{2012}}
\revised{\smonth{12} \syear{2013}}

%
\begin{abstract}
We consider the random walk Metropolis algorithm on $\mathbb{R}^n$
with Gaussian
proposals, and when the target probability measure is the $n$-fold
product of a one-dimensional law. In the limit
$n \to\infty$, it is well known (see
[\textit{Ann. Appl. Probab.} \textbf{7} (1997) 110--120])
that, when the variance of the
proposal scales inversely proportional to the dimension $n$ whereas
time is accelerated by the factor $n$, a
diffusive limit is obtained for each component of the Markov chain if
this chain starts at equilibrium. This paper
extends this result when the initial distribution is not the
target probability measure. Remarking that the interaction between the
components of the chain due to the
common acceptance/rejection of the proposed moves is of mean-field
type, we obtain a propagation of chaos
result under the same scaling as in the stationary case. This proves
that, in terms of the dimension $n$,
the same scaling holds for the transient phase of the
Metropolis--Hastings algorithm as near stationarity.
The diffusive and mean-field limit of each component is a diffusion
process nonlinear in the sense of
McKean. This opens the route to new investigations of the optimal
choice for the variance of the proposal distribution in order to accelerate
convergence to equilibrium (see [Optimal scaling for the transient phase of Metropolis--Hastings
algorithms: The longtime behavior \textit{Bernoulli} (2014) To appear]).
\end{abstract}

%
\end{frontmatter}

\setcounter{footnote}{1}
\section{Introduction}\label{secintro}
Many Markov Chain Monte Carlo (MCMC) methods are based on the
Metropolis--Hastings algorithm~\cite
{metropolis-rosenbluth-rosenbltuh-teller-teller-53,hastings-70}. Let us
recall this well-known sampling technique. Let us consider a target
probability distribution on $\R^n$ with density $p$. Starting from an
initial random variable $X_0$, the Metropolis--Hastings algorithm
generates iteratively a Markov chain $(X_k)_{k \ge0}$ in two steps. At
time $k$, given $X_k$, a candidate $Y_{k+1}$ is sampled using a
proposal distribution with density $q(X_k,y)$. Then the proposal
$Y_{k+1}$ is accepted with probability $\alpha(X_k,Y_{k+1})$, where
\[
\alpha(x,y)= 1\wedge\frac{p(y)q(y,x)}{p(x)q(x,y)}.
\]
Here and in the following, we use the standard notation $a \wedge b =
\min(a,b)$. If the
proposed value is accepted, then $X_{k+1}=Y_{k+1}$ otherwise
$X_{k+1}=X_{k}$. The Markov chain $(X_k)_{k \ge0}$ is by construction
reversible with respect to the target density $p$, and thus admits
$p(x) \,dx$ as an invariant distribution. The efficiency of this algorithm
highly depends on the choice of the proposal distribution $q$. One
common choice is a Gaussian proposal centered at the current position
$x \in\R^n$
with variance~$\sigma^2 \Id_{n \times n}$:
\[
q(x,y)=\frac{1}{(2 \pi\sigma^2)^{n/2}} \exp \biggl(-\frac
{\llvert x-y\rrvert ^2}{2\sigma^2} \biggr).
\]
Since the proposal is symmetric ($q(x,y)=q(y,x)$), the acceptance
probability reduces to
%
\begin{equation}
\label{defacceptanceRWM} \alpha(x,y)= 1\wedge\frac{p(y)}{p(x)}.
\end{equation}
Metropolis--Hastings algorithms with symmetric kernels are called
random walk Metropolis (RWM) algorithms.

The choice of the variance $\sigma^2$ is crucial for the performance
of the RWM algorithm. It should be sufficiently large to ensure a good
exploration
of the state space, but not too large otherwise the rejection rate
becomes typically very high since the proposed moves fall in low
probability regions, in particular in high dimension. It is expected
that the higher the dimension, the smaller the variance of the
proposal should be. The first theoretical results to optimize the
choice of $\sigma^2$ in terms of the dimension $n$ are due to
Roberts, Gelman and Gilks in~\cite{roberts-gelman-gilks-97}. The authors study
the RWM algorithm under two fundamental (and somewhat restrictive)
assumptions: (i) the target probability distribution is the $n$-fold tensor
product of a one-dimensional density:
%
\begin{equation}
\label{eqpi} p(x) = \prod_{i=1}^n
\frac{\exp(-V(x_i))}{Z},
\end{equation}
where $x=(x_1, \ldots, x_n)$ and $Z=\int_{\R} \exp
(-V)$, and (ii) the initial distribution is the target probability:
\[
X^n_0 \sim p(x) \,dx.
\]
The superscript $n$ in the Markov chain $(X^n_k)_{ k \ge0}$
explicitly indicates the dependency on the dimension $n$.
Then, under additional regularity assumptions on $V$, the authors prove
that for a proper scaling of the variance as a function of the
dimension, namely
\[
\sigma_n^2=\frac{l^2}{n},
\]
where $l$ is a fixed constant, the Markov process $
(X^{1,n}_{\lfloor nt \rfloor} )_{t \ge0}$ (where\vspace*{1pt} $X^{1,n}_k \in
\R$ denotes the first component of $X^n_k \in\R^n$) converges in law
to a diffusion process:
%
\begin{equation}
\label{eqCasEq} dX_t = \sqrt{ h(l)} \,dB_t - h(l)
\tfrac{1}{2} V' (X_t) \,dt,
\end{equation}
where $(B_t)_{t\geq0}$ is a standard Brownian motion,
%
\begin{equation}
\label{eqh} h(l)=2l^2 \Phi \biggl( - \frac{l \sqrt{I}}{2} \biggr)
\quad\mbox{and}\quad I=\int_{\R} \bigl(V'
\bigr)^2 \frac{\exp(-V)}{Z}.
\end{equation}
Here and in the following, $\lfloor\cdot\rfloor$ denotes the integer
part (for $y \in\R$, $\lfloor y \rfloor\in\Z$ and $\lfloor y
\rfloor\le y < \lfloor y \rfloor+1$) and $\Phi$ is the cumulative
distribution function of the normal distribution [$\Phi(x) =
\frac{1}{\sqrt{2 \pi}} \int_{-\infty}^x \exp(-y^2/2) \,dy$]. The
scaling as a function of the dimension of the variance and of the
time are indications on how to make the RWM algorithm
efficient in high dimension. Moreover, a
practical counterpart of this result is that $l$ should be chosen such
that $h(l)$ is maximum (the optimal value of $l$ is
$l^*=\frac{2.38}{\sqrt{I}}$), in order to optimize the time
scaling in~(\ref{eqCasEq}). This optimal value of $l$ corresponds
equivalently to an average
acceptance rate $0.234$ (independently of the value of $I$): for $l=l^*$,
\[
\int\!\!\int\alpha(x,y) p(x) q(x,y) \,dx \,dy = 2 \Phi \biggl(-\frac{l^* \sqrt{I}}{2}
\biggr) \simeq0.234.
\]
Thus, the practical way to choose $\sigma^2$ is to scale it in such a
way that the average acceptance rate is roughly $1/4$.

There exist several extensions of such results for various
Metropolis--Hastings algorithms, see
\cite
{roberts-rosental-97,roberts-rosenthal-01,neal-roberts-11,neal-roberts-yuen-12,bedard-douc-moulines-12,bedard-douc-moulines-12b,beskos-roberts-sanz-serna-stuart-12},
and some of them relax in particular the first main assumption
mentioned above about
the product form of the target distribution;
see~\cite
{breyer-piccioni-scarlatti-04,breyer-roberts-00,bedard-07,bedard-08,beskos-roberts-stuart-09}.
Extensions to infinite-dimensional settings
have also been explored; see~\cite
{mattingly-pillai-stuart-12,pillai-struart-thiery-11,beskos-roberts-stuart-09}.

All these results assume stationarity: the initial
measure is the target probability measure. To the best of the authors'
knowledge, the only works which deal with a nonstationary case
are~\cite{christensen-roberts-rosenthal-05} where partial scaling
results are obtained for the RWM algorithm with a Gaussian target and
\cite{PST11}. In the latter paper, the target measure is assumed
to be absolutely continuous with respect to the law of an infinite-dimensional Gaussian random field and this measure is approximated in
a space of dimension $n$ where the MCMC algorithm is performed. The
authors consider a modified RWM algorithm (called preconditioned
Crank--Nicolson walk) started at a deterministic initial condition and
prove that when $\sigma_n$ tends to $0$ as $n$ tends to $\infty$ (with
no restriction on the rate of convergence of $\sigma_n$ to $0$), the
rescaled algorithm converges to a stochastic partial differential equation,
started at the same initial condition.

The aim of the present article is to show that, for the RWM algorithm, using
the same scaling for the variance and the time as in the
stationary case [namely $\sigma_n^2=\frac{l^2}{n}$ and considering
$ (X^{1,n}_{\lfloor nt \rfloor} )_{t \ge0}$], one\vspace*{2pt} obtains in
the limit $n$ goes to infinity the nonlinear (in the sense of McKean)
diffusion process:
%
\begin{eqnarray}\label{eqSDENL}
dX_t &=& \Gamma^{1/2} \bigl(\E \bigl[
\bigl(V'(X_t) \bigr)^2 \bigr],\E
\bigl[V''(X_t) \bigr] \bigr)
\,dB_t
\nonumber\\[-8pt]\\[-8pt]\nonumber
&&{} - {\cG} \bigl(\E \bigl[ \bigl(V'(X_t)
\bigr)^2 \bigr],\E \bigl[V''(X_t)
\bigr] \bigr)V'(X_t) \,dt,
\end{eqnarray}
where, for $a \in[0,+\infty]$ and $b \in\R$,
%
\begin{equation}
\label{eqGamma} \Gamma(a,b)= %
\cases{
\displaystyle l^2\Phi \biggl(-\frac{l b}{2\sqrt{a}} \biggr)+l^2e^{(l^2(a-b))/2}\Phi \biggl(l \biggl(\frac{b}{2\sqrt
{a}}-\sqrt{a} \biggr) \biggr),
\vspace*{3pt}\cr
\hspace*{78.5pt}\mbox{if $a\in(0,+\infty)$,}
\vspace*{3pt}\cr
\displaystyle\frac{l^2}{2},\hspace*{62.5pt}\mbox{if $a=+\infty$,}
\vspace*{3pt}\cr
\displaystyle l^2 e^{-(l^2b^+)/2},\qquad\mbox{if $a=0$},}
\end{equation}
where $b^+=\max(b,0)$, and
%
\begin{equation}
\label{eqG}
\qquad {\mathcal G}(a,b)= \cases{
\displaystyle l^2e^{(l^2(a-b))/2}\Phi \biggl(l \biggl(\frac{b}{2\sqrt{a}}-\sqrt{a} \biggr) \biggr), &\quad if $a\in(0,+\infty)$,
\vspace*{3pt}\cr
0, &\quad if $a=+\infty$,
\vspace*{3pt}\cr
\displaystyle 1_{\{b>0\}}l^2 e^{-(l^2b)/2}, &\quad if $a=0$.}
\end{equation}
Notice that we will assume $V''$ to be bounded, so that the
coefficients in~(\ref{eqSDENL}) are well defined. This convergence
result is
precisely stated in Theorem~\ref{chaos} below and can be seen as a
mean-field limit combined with a diffusion approximation. We would like
to mention that another (different in
nature) mean-field limit is considered in~\cite
{breyer-piccioni-scarlatti-04} in the
context of optimal scaling: the limit is obtained, under the
stationarity assumption,
for a target measure which admits some mean-field limit as $n \to
\infty$.

Our convergence result generalizes the
previous analysis in~\cite{roberts-gelman-gilks-97} which was limited
to the stationary
case [namely $X^n_0$ is distributed according to~$p(x) \,dx$]. In particular,
in the stationary case, we recover the dynamics~(\ref{eqCasEq}). It
also generalizes results from~\cite{christensen-roberts-rosenthal-05}
to non-Gaussian
targets.

The
proof is based on a classical technique to prove propagation of
chaos \cite{sznitman-91}. We first show the tightness of the empirical
distribution. Then we pass to the limit in a martingale problem,
which is the weak formulation of~(\ref{eqSDENL}). Notice that such a
weak formulation has also recently been used in~\cite
{mattingly-pillai-stuart-12} to deal
with the stationary case.

This new result opens the
route to new investigations of the optimal choice for the variance of
the proposal
distribution, by precisely taking into account the transient regime
(when the Markov chain is not yet at equilibrium). It shows, for
example, how to scale properly the variance and the number of samples
as a function of the dimension, at least for a product target. A more
detailed analysis of the longtime behavior of the nonlinear
diffusion~(\ref{eqSDENL}) and of the practical counterparts of this
convergence result are the
subject of a companion paper~\cite{JLM2}.

The paper is organized as follows. In Section~\ref{secmain}, we state
our main convergence result, we present a formal derivation of the
limiting diffusion process and we explain the three main steps of its
rigorous proof. Sections~\ref{secmart},~\ref{sectight} and~\ref
{seclimit} are, respectively, devoted to each of these main steps:
uniqueness for the stochastic differential equation (\ref{eqSDENL})
and its weak formulation as a martingale problem, tightness of the laws
of the processes $ (X^{1,n}_{\lfloor nt \rfloor} )_{t \ge
0}$ and identification of the limit\vspace*{1pt} probability measures on the path
space thanks to the martingale problem. Last, in Section~\ref
{secprop}, we prove the convergence of the acceptance probability in
the RWM algorithm to
$\frac{1}{l^2}\Gamma(\E[(V'(X_t))^2],\E[V''(X_t)])$.

\section{The main convergence result}\label{secmain}
Let us first present the precise statement for the main convergence
result. Then we will give a formal derivation of the limiting
process before sketching the rigorous proof.

\subsection{Notation and convergence to the diffusion process}

We consider a random walk Metropolis algorithm using Gaussian
proposal with variance $\sigma_n^2=\frac{l^2}{n}$, and with target
$p$ defined by~(\ref{eqpi}). The Markov
chain generated using this algorithm writes
%
\begin{equation}
\label{rwm} X^{i,n}_{k+1}=X^{i,n}_k+
\frac{l}{\sqrt{n}}G^i_{k+1}1_{{\mathcal
A}_{k+1}},\qquad 1\leq i\leq n
\end{equation}
with
\[
{\mathcal A}_{k+1}= \bigl\{U_{k+1}\leq e^{\sum_{i=1}^n
(V(X^{i,n}_k)-V(X^{i,n}_k+({l}/{\sqrt{n}})G^i_{k+1}))}
\bigr\},
\]
where $(G^i_k)_{i,k\geq1}$ is a sequence of independent and
identically distributed (i.i.d.) normal random
variables, independent from a sequence $(U_k)_{k\geq1}$ of i.i.d.
random variables uniform on $[0,1]$. We assume that the initial
positions $(X^{1,n}_0,\ldots,X^{n,n}_0)$ are exchangeable
(namely the law of the vector is invariant under \mbox{permutation} of the indices)
and independent from $(G^i_k)_{i,k\geq1}$ and $(U_k)_{k\geq1}$.
Exchangeability is preserved by the dynamics: for all $k
\ge1$, $(X^{1,n}_k,\ldots,X^{n,n}_k)$ are exchangeable. We denote by
${\mathcal F}^n_k$ the sigma field generated by
$(X^{1,n}_0,\ldots,X^{n,n}_0)$ and $(G^1_l,\ldots,G^n_l,U_l)_{1\leq
l\leq k}$.\vspace*{1pt}

In all the following, we also assume that
%
\begin{equation}
\label{eqhypV} \cases{\mbox{$V$ is a ${\mathcal C}^3$ function on $
\R$}
\cr
\mbox{with bounded second- and third-order derivatives.}}
\end{equation}

For $t>0$ and $i \in\{1,\ldots,n\}$, let
\begin{eqnarray*}
Y^{i,n}_t&=& \bigl(\lceil nt\rceil-nt \bigr)X^{i,n}_{\lfloor
nt\rfloor}+
\bigl(nt-\lfloor nt\rfloor \bigr)X^{i,n}_{\lceil
nt\rceil}
\\
&=&X^{i,n}_{\lfloor nt\rfloor}+ \bigl(nt-\lfloor nt\rfloor \bigr)
\frac{l}{\sqrt{n}}G^i_{\lceil nt\rceil}1_{{\mathcal
A}_{\lceil nt\rceil}}
\end{eqnarray*}
be the linear
interpolation of the Markov chain obtained by rescaling time
(the characteristic time scale is $1/n$, and
$Y^{i,n}_{k/n}=X^{i,n}_k$, $\forall k \in\Z$). Here\vspace*{1pt} and in the following
$\lceil\cdot\rceil$ is the upper integer part (for $y \in\R$,
$\lceil y \rceil\in\Z$ and $\lceil y \rceil-1 < y \le\lceil y
\rceil$).

Let us define the notion of convergence (namely the propagation of
chaos) that will be useful to study
the convergence of the interacting particle system
$((Y^{1,n}_t, \ldots, Y^{n,n}_t )_{t \ge0})_{ n \ge1}$ in the limit
$n$ goes to infinity.

\begin{definition}
Let $E$ be a separable metric space. A sequence
$(\chi^n_1,\ldots,\break \chi_n^n)_{n\geq1}$ of exchangeable $E^n$-valued
random variables is said to be $\nu$-chaotic where $\nu$ is a
probability measure on $E$ if for fixed $j\in\tN^*$, the law of
$(\chi^n_1,\ldots,\chi_j^n)$ converges in distribution to $\nu
^{\otimes j}$ as
$n$ goes to $\infty$.
\end{definition}

We are now in position to state the main convergence result.

\begin{theorem}\label{chaos}Assume~(\ref{eqhypV}) and let $m$ be a
probability measure on $\R$
such that $\int_{\R} (V')^4(x) m(dx) <+\infty$. If the initial
positions $(X^{1,n}_0,\ldots,X^{n,n}_0)_{n\geq1}$ are exchangeable,
$m$-chaotic and such that
$\sup_n\E[(V'(X^{1,n}_0))^4]<+\infty$,
then the processes $((Y^{1,n}_t,\ldots,Y^{n,n}_t)_{t\geq0})_{n\geq
1}$ are $P$-chaotic where $P$ denotes the law [on the space
${\mathcal C}(\R_+,\R)$ of continuous functions with values in $\R$]
of the solution to the nonlinear stochastic differential equation in
the sense of McKean (for which strong and weak existence and uniqueness hold)
%
\begin{eqnarray}
\label{edsnonlin} X_t&=&\xi+\int_0^t
\Gamma^{1/2} \bigl(\E \bigl[ \bigl(V'(X_s)
\bigr)^2 \bigr],\E \bigl[V''(X_s)
\bigr] \bigr)\,dB_s
\nonumber
\\[-8pt]
\\[-8pt]
&&{} -\int_0^t{\cG} \bigl(\E \bigl[
\bigl(V'(X_s) \bigr)^2 \bigr],\E
\bigl[V''(X_s) \bigr]
\bigr)V'(X_s)\,ds,
\nonumber
\end{eqnarray}
where $\Gamma$ and $\cG$ are, respectively, defined by~(\ref{eqGamma})
and~(\ref{eqG}) and $(B_t)_{t\geq1}$ is a Brownian motion
independent from the initial position $\xi$ distributed according to
$m$.
\end{theorem}

Let us make a few remarks on this result. First, concerning the
assumption on the initial positions
$(X^{1,n}_0,\ldots,X^{n,n}_0)_{n\geq1}$, we note that it is satisfied,
for instance,
when the random variables $X^{1,n}_0,\ldots,X^{n,n}_0$ are
i.i.d. according to the probability measure~$m$ on $\R$. Second,
notice that the results of Theorem~\ref{chaos} do not require $\exp(-V)$
to be integrable. Finally, according to \cite{ethier-kurtz-86} (see
Proposition~10.4, page~149 and Theorem~10.2, page~148), under the assumptions
of Theorem~\ref{chaos}, the piecewise constant processes
$((X^{1,n}_{\lfloor nt\rfloor},\ldots,X^{n,n}_{\lfloor nt\rfloor
})_{t\geq0})_{n\geq1}$ are\vspace*{1pt} also $P$-chaotic when the space of c\`
{a}dl\`{a}g sample paths from $[0,+\infty)$ is endowed with the
topology of uniform convergence on compact sets.

In addition to the previous convergence result, we are able to
identify the limiting average acceptance rate.%

\begin{proposition}\label{proplimiting-acceptance} Under the
assumptions of Theorem~\ref{chaos}, the function
\[
t\mapsto \E \biggl\llvert \mathbb{P} \bigl({\mathcal A}_{\lfloor nt\rfloor+1}|{
\mathcal F}^n_{\lfloor
nt\rfloor} \bigr)-\frac{1}{l^2}\Gamma \bigl(\E
\bigl[ \bigl(V'(X_t) \bigr)^2 \bigr],\E
\bigl[V''(X_t) \bigr] \bigr) \biggr\rrvert
\]
converges locally uniformly to $0$ and in particular, the average
acceptance rate $t\mapsto
\mathbb{P}({\mathcal A}_{\lfloor nt\rfloor+1})$ converges locally uniformly
to $t\mapsto\acc(\E[(V'(X_t))^2],\break \E[V''(X_t)])$ where for any
$a\geq0$ and $b\in\R$,
%
\begin{equation}
\label{deflimitacc} \acc(a,b)=\frac{\Gamma(a,b)}{l^2}.
\end{equation}
\end{proposition}

In the following, we will also need the infinitesimal generator associated
to~(\ref{edsnonlin}). For a probability measure $\mu$ on $\R$,
$\langle\mu,V''\rangle$ is well defined by boundedness of $V''$, and
$\langle\mu,(V')^2\rangle$ is also well defined in $[0,+\infty]$.
Here and in the following,
the bracket notation refers to the duality bracket for
probability measures on $\R$: for~$\mu$ a probability measure and
$\phi$ a bounded or positive measurable function,
\[
\langle\mu, \phi\rangle= \int_{\R} \phi(x) \mu(dx).
\]
The infinitesimal generator associated to~(\ref{edsnonlin}) is $L_\mu
$ defined by
%
\begin{eqnarray}\label{deflmu}
L_\mu\varphi(x)
&=& \tfrac{1}{2}\Gamma \bigl( \bigl\langle\mu,
\bigl(V' \bigr)^2 \bigr\rangle, \bigl\langle
\mu,V'' \bigr\rangle \bigr)\varphi''(x)
\nonumber\\[-8pt]\\[-8pt]\nonumber
&&{} -\cG \bigl( \bigl\langle\mu, \bigl(V' \bigr)^2 \bigr
\rangle, \bigl\langle\mu,V'' \bigr\rangle
\bigr)V'(x) \varphi'(x).
\end{eqnarray}
More precisely, if $(X_t)_{t\geq0}$ satisfies~(\ref{edsnonlin}) and
$P_t$ denotes the law of $X_t$, then
%
\begin{eqnarray}\label{eqmartin}
&& \mbox{for any test function } \varphi,
\nonumber\\[-8pt]\\[-8pt]\nonumber
&& \qquad \biggl(\varphi(X_t) - \int
_0^t L_{P_s} \varphi(X_s)
\,ds \biggr)_{t\geq0}\mbox{ is a martingale.}
\end{eqnarray}
Equivalently, for any $s < t$,
%
\begin{equation}
\label{eqmartingale} \E \biggl( \varphi(X_t) - \int_s^t
L_{P_r} \varphi(X_r) \,dr \Big| {\mathcal F}_s
\biggr) = \varphi(X_s),
\end{equation}
where ${\mathcal F}_s=\sigma(X_r, r \le s)$.
Actually, as explained in Section~\ref{secmart} below, this martingale
representation characterizes the distribution [over ${\mathcal C}(\R
_+,\R)$] of solutions to~(\ref{edsnonlin}): probability measures
under which (\ref{eqmartin}) holds are distributions of solutions
to~(\ref{edsnonlin}), and reciprocally.

\subsection{Relation to previous results in the literature}

Let us discuss how this theorem is related to previous results in
the literature. First, when $Z=\int_\R e^{-V(x)}\,dx<+\infty$, our
convergence result generalizes the scaling limit for the
random walk Metropolis--Hastings algorithm stated in the early
paper~\cite{roberts-gelman-gilks-97} under\vspace*{1pt}
the restrictive assumption that the vector of initial positions
$(X^{1,n}_0,\ldots,X^{n,n}_0)$ is distributed according to the target
distribution $p(x) \,dx$. In this case, it is clear that for all
$n,k\in{\mathbb N}$, $(X^{1,n}_k,\ldots,X^{n,n}_k)$ is distributed
according to $p(x) \,dx$. Moreover, we have the following result.

\begin{lemma}
Assume that~(\ref{eqhypV}) holds, and that $\int_\R e^{-V(x)}\,dx <
\infty$.
Then
\[
\int_\R \bigl(V'(x) \bigr)^2e^{-V(x)}\,dx=
\int_\R V''(x)e^{-V(x)}\,dx<+
\infty.
\]
\end{lemma}

\begin{pf}
The integrability of $e^{-V}$ implies that $\liminf_{\llvert x\rrvert \to\infty
}\llvert x\rrvert e^{-V(x)}=0$. Since $\llvert V'(x)\rrvert \leq\llvert V'(0)\rrvert +\llVert  V''\rrVert  _\infty\llvert x\rrvert $, one
deduces the existence of a sequence $(x_n)_n$ of negative numbers
tending to $-\infty$ and a
sequence $(y_n)_n$ of positive numbers tending to $+\infty$ such
that $\lim_{n\to
+\infty}\llvert V'(x_n)\rrvert e^{-V(x_n)}+\llvert V'(y_n)\rrvert e^{-V(y_n)}=0$. By integration
by parts,
\begin{eqnarray*}
&& \int_{x_n}^{y_n} \bigl(V'(x)
\bigr)^2e^{-V(x)}\,dx
\\
&&\qquad = V'(x_n)e^{-V(x_n)}-V'(y_n)e^{-V(y_n)}+
\int_{x_n}^{y_n}V''(x)e^{-V(x)}\,dx.
\end{eqnarray*}
Taking the limit $n\to\infty$ thanks to monotone convergence in the
left-hand side and thanks to Lebesgue's theorem and boundedness of
$V''$ in the integral in the right-hand side, one concludes that $\int_\R(V'(x))^2e^{-V(x)}\,dx=\int_\R V''(x)e^{-V(x)}\,dx<+\infty$.
\end{pf}

One deduces that for each $t\geq0$ the solution
$X_t$ of (\ref{edsnonlin}) is distributed according to $Z^{-1}
\exp(-V(x)) \,dx$ so that $(X_t)_{t\geq0}$ also solves the
stochastic differential equation~(\ref{eqCasEq})--(\ref{eqh}) with
time-homogeneous coefficients [here,\vspace*{1pt} we use the fact that
$\Gamma(I,I)=2\cG(I,I)=h(l)$\vspace*{1pt} where $I=\int_\R
(V'(x))^2e^{-V(x)}\frac{dx}{Z}=\int_\R
V''(x)e^{-V(x)}\frac{dx}{Z}$]. Notice that our convergence
result requires more regularity but less\vspace*{1.5pt} integrability than in
\cite{roberts-gelman-gilks-97}, Theorem 1.1, where the log-density $-V$
is assumed to be
${\mathcal C}^2$ with a bounded second-order derivative and such that
$\int_{\R}
(V')^8 \exp(-V)<+\infty$.

Second, we also recover results from~\cite
{christensen-roberts-rosenthal-05}, where the authors
consider a nonstationary case, but restrict their analysis to
Gaussian distributions: $V(x)=\frac{x^2}{2}$. In this case, the
function $V''$ is
constant equal to $1$ and, for $X_t$ solution to~(\ref{edsnonlin}), one
obtains that
\begin{eqnarray*}
\frac{d}{dt}\E \bigl[X_t^2 \bigr]&=&\Gamma \bigl(\E
\bigl[X_t^2 \bigr],1 \bigr) -2\E \bigl[X_t^2
\bigr]{\cG} \bigl(\E \bigl[X_t^2 \bigr],1 \bigr)
\\
&=&l^2\Phi \biggl(-\frac{l}{2\sqrt{\E[X_t^2]}} \biggr)
\\
&&{} + \bigl(1-2 \E
\bigl(X_t^2 \bigr) \bigr)l^2e^{(l^2(\E[X_t^2]-1))/2}
\Phi \biggl(l \biggl(\frac{1}{2\sqrt{\E[X_t^2]}}-\sqrt{\E \bigl[X_t^2
\bigr]} \biggr) \biggr).
\end{eqnarray*}
This is indeed the ordinary differential equation satisfied by the
deterministic function obtained as the limit (when $n \to\infty$) of
the processes $ (\frac{1}{n}\sum_{i=1}^n(X^{i,n}_{\lfloor
nt\rfloor})^2 )_{t\geq0}$ in~\cite{christensen-roberts-rosenthal-05}, Theorem~1. More precisely, the proof of our
Proposition~\ref{proplimiting-acceptance} ensures that $\E\llvert \frac{1}{n}\sum_{i=1}^n(X^{i,n}_{\lfloor
nt\rfloor})^2-\E[X_t^2]\rrvert $ converges to $0$ locally uniformly
in $t$ as $n\to\infty$.

\subsection{A formal derivation}\label{secformal}

Before going into the details of a rigorous proof, let us explain how
this limit diffusion process can be formally derived.

First, let us
make precise how to choose the scaling of $\sigma_n$ as a function of
$n$. The idea (see~\cite{roberts-rosenthal-01}) is to
choose $\sigma_n$ in such a way that the limiting acceptance rate
(when $n \to\infty$) is neither zero nor one. In the first case, this
would mean that the variance of the proposal is too large, so that all
proposed moves are rejected. In the second case, the variance of the
proposal is too small, and the rate of convergence to equilibrium is
thus not optimal. In particular, it is easy to check that $\sigma_n$
should go to zero as $n$ goes to infinity. Now, notice that the
limiting acceptance rate is
%
\begin{eqnarray}\label{eqacceptRWM}
\E \bigl( 1_{{\mathcal
A}_{k+1}} | {\mathcal F}^n_k \bigr)
&=&\E \bigl( e^{\sum_{i=1}^n
(V(X^{i,n}_k)-V(X^{i,n}_k+\sigma_n G^i_{k+1}))} \wedge1 | {\mathcal F}^n_k
\bigr)\nonumber
\\
&=& \E \bigl(e^{- \sum_{i=1}^n
(V'(X^{i,n}_k) \sigma_n G^i_{k+1} +
V''(X^{i,n}_k) ({\sigma_n^2}/{2}) ) } \wedge1 | {\mathcal F}^n_k\bigr)\nonumber
\\
&&{} + {\mathcal O} \bigl(n \sigma_n^{3} \bigr) + {
\mathcal O} \bigl(\sqrt{n} \sigma_n^{2} \bigr)
\nonumber\\[-8pt]\\[-8pt]\nonumber
&=& \exp \biggl(\frac{a_n-b_n}{2} \biggr) \Phi \biggl( \frac{b_n}{2
\sqrt{a_n}} -\sqrt{a_n} \biggr) + \Phi \biggl(-\frac{b_n}{2\sqrt{a_n}} \biggr)
\\
&&{} + {\mathcal O} \bigl(n \sigma_n^{3} \bigr) + {\mathcal O}
\bigl( \sqrt{n} \sigma_n^{2} \bigr)
\nonumber
\\
&=&\frac{1}{l^2} \Gamma(a_n,b_n) + {\mathcal O}
\bigl(n \sigma_n^{3} \bigr) + {\mathcal O} \bigl(\sqrt{n}
\sigma_n^{2} \bigr),\nonumber
\end{eqnarray}
where\vspace*{1.5pt} $a_n=\frac{\sigma_n^2}{l^2} \sum_{i=1}^n
(V'(X^{i,n}_k))^2 $ and $b_n=\frac{\sigma_n^2}{l^2} \sum_{i=1}^n
V''(X^{i,n}_k)$. To obtain~(\ref{eqacceptRWM}), we used an explicit
computation of the
expectation with respect to the Gaussian measure; see~(\ref{calesp3})
below (with $\alpha=0$).
From this expression, assuming a propagation of chaos (law of large
number) result on the random variables $(X^{i,n}_k)_{1 \le i \le n}$,
one can check that the correct scaling for the variance is
$\sigma_n^2=\frac{l^2}{n}$ in order to obtain a nontrivial
limiting acceptance rate (see
Proposition~\ref{proplimiting-acceptance} above). More
precisely, if $a_n \to0$ and $b_n \to0$, then the acceptance
rate goes to $1$ [by continuity of $\Gamma$ at point $(0,0)$, see
Lemma~\ref{propggam} below]. If $a_n \sim\alpha n^\epsilon$ and $b_n
\sim\beta n^\epsilon$ (for some $\epsilon>0$), then the acceptance
rate goes to~$0$ if $\beta>0$ and to $1$ if $\beta< 0$.\vspace*{1pt}

Using the scaling $\sigma_n^2=\frac{l^2}{n}$, we observe that,
for a test function $\varphi\dvtx  \R\to\R$,
%
\begin{eqnarray}\label{eqgeninf}
\E \bigl(\varphi \bigl(X^{1,n}_{k+1} \bigr) | {\mathcal
F}^n_k \bigr)
&=& \E \biggl(\varphi \biggl(X^{1,n}_k+\frac{l}{\sqrt
{n}}G^1_{k+1}1_{{\mathcal
A}_{k+1}}
\biggr) \Big| {\mathcal F}^n_k \biggr)
\nonumber
\\
&=&  \varphi \bigl(X^{1,n}_k \bigr) + \varphi'
\bigl(X^{1,n}_k \bigr) \frac{l}{\sqrt{n}} \E \bigl(
G^1_{k+1}1_{{\mathcal
A}_{k+1}} | {\mathcal F}^n_k
\bigr)
\\
&&{} +\frac{l^2}{2n} \varphi'' \bigl(X^{1,n}_k
\bigr) \E \bigl( \bigl(G^1_{k+1} \bigr)^2
1_{{\mathcal
A}_{k+1}} | {\mathcal F}^n_k \bigr) + {\mathcal
O} \bigl(n^{-3/2} \bigr).\nonumber
\end{eqnarray}
We compute
%
\begin{eqnarray} \label{eqdrift}
&& \E \bigl( G^1_{k+1} 1_{{\mathcal
A}_{k+1}} | {\mathcal
F}^n_k \bigr)\nonumber
\\
&&\qquad = \E \bigl( G^1_{k+1}
\bigl( e^{\sum_{i=1}^n
(V(X^{i,n}_k)-V(X^{i,n}_k+({l}/{\sqrt{n}})G^i_{k+1}))} \wedge1 \bigr) | {\mathcal F}^n_k
\bigr)\nonumber
\\
&&\qquad = \E \bigl( G^1_{k+1} \bigl( e^{- \sum_{i=1}^n
(V'(X^{i,n}_k) ({l}/{\sqrt{n}})G^i_{k+1} +
V''(X^{i,n}_k) ({l^2}/{(2n)}) ) } \wedge1 \bigr) | {\mathcal F}^n_k \bigr)
\\
&&\quad\qquad{}  + {\mathcal O} \bigl(n^{-1/2} \bigr)\nonumber
\\
&&\qquad =- V' \bigl(X^{1,n}_k \bigr)
\frac{1}{l \sqrt{n}} \cG \bigl( \bigl\langle\nu ^n_{k},
\bigl(V' \bigr)^2 \bigr\rangle, \bigl\langle
\nu^n_{k}, V'' \bigr\rangle
\bigr) + {\mathcal O} \bigl(n^{-1/2} \bigr),\nonumber
\end{eqnarray}
where
\[
\nu^n_{k}=\frac{1}{n} \sum
_{i=1}^n \delta_{X^{i,n}_k}
\]
denotes the empirical distribution associated to the interacting
particle system.
Equation~(\ref{eqdrift}) is a consequence of~(\ref{calesp1})
below. A more detailed analysis (see Lemma~\ref{propdrift}
below) shows that the remainder is of order ${\mathcal
O}(n^{-3/4})$. This is one of the most crucial estimate to prove
rigorously the convergence result. For the diffusion term, we get
%
\begin{eqnarray}\label{eqdiff}
\quad && \E \bigl( \bigl(G^1_{k+1} \bigr)^2
1_{{\mathcal
A}_{k+1}} | {\mathcal F}^n_k \bigr)\nonumber
\\
&&\qquad  = \E \bigl(
\bigl(G^1_{k+1} \bigr)^2 \bigl( e^{\sum_{i=1}^n
(V(X^{i,n}_k)-V(X^{i,n}_k+({l}/{\sqrt{n}})G^i_{k+1}))}
\wedge1 \bigr) | {\mathcal F}^n_k \bigr)
\nonumber
\\
&&\qquad = \E \bigl( \bigl(G^1_{k+1} \bigr)^2 \bigl(
e^{- \sum_{i=1}^n
(V'(X^{i,n}_k) ({l}/{\sqrt{n}})G^i_{k+1} +
V''(X^{i,n}_k) ({l^2}/{(2n)}))} \wedge1 \bigr) | {\mathcal F}^n_k
\bigr)
\\
&&\quad\qquad{} + {\mathcal O} \bigl(n^{-1/2} \bigr)
\nonumber
\\
&&\qquad = \frac{1}{l^2} \Gamma \bigl( \bigl\langle\nu^n_{k},
\bigl(V' \bigr)^2 \bigr\rangle, \bigl\langle
\nu^n_{k}, V'' \bigr\rangle
\bigr) + {\mathcal O} \bigl(n^{-1/2} \bigr).\nonumber
\end{eqnarray}
To obtain~(\ref{eqdiff}), we again used an explicit computation;
see~(\ref{calesp3}) below.

By plugging~(\ref{eqdrift}) [with the remainder of order ${\mathcal
O}(n^{-3/4})$] and~(\ref{eqdiff})
into~(\ref{eqgeninf}), we see that the correct scaling in time is
to consider $Y^{i,n}_t$ such that $Y^{i,n}_{k/n}=X^{i,n}_k$, and we get
%
\begin{eqnarray}\label{pdmdisc}
&& \E \bigl( \varphi \bigl(Y^{1,n}_{(k+1)/n} \bigr) | {\mathcal
F}^n_k \bigr)\nonumber
\\
&&\qquad
= \varphi \bigl(Y^{1,n}_{k/n}
\bigr)  - \varphi' \bigl(Y^{1,n}_{k/n} \bigr)
\frac{1}{n} V' \bigl(Y^{1,n}_{k/n} \bigr)
\cG \bigl( \bigl\langle\mu^n_{k/n}, \bigl(V'
\bigr)^2 \bigr\rangle, \bigl\langle\mu^n_{k/n},
V'' \bigr\rangle \bigr)
\nonumber\\[-8pt]\\[-8pt]\nonumber
&&\quad\qquad{} +\frac{1}{2n} \varphi'' \bigl(
Y^{1,n}_{k/n} \bigr) \Gamma \bigl( \bigl\langle\mu
^n_{k/n}, \bigl(V' \bigr)^2 \bigr
\rangle, \bigl\langle\mu^n_{k/n}, V''
\bigr\rangle \bigr) + {\mathcal O} \bigl(n^{-5/4} \bigr)
\\
&&\qquad = \varphi \bigl(Y^{1,n}_{k/n} \bigr) + \frac{1}{n} (
L_{
\mu^n_{k/n}} \varphi ) \bigl(Y^{1,n}_{k/n} \bigr) + {
\mathcal O} \bigl(n^{-5/4} \bigr),
\nonumber
\end{eqnarray}
where $L_\mu$ is defined by~(\ref{deflmu}), and $\mu^n_t$ denotes the
time-marginal of $\mu^n$ defined by~(\ref{eqmun}) below (for $k \in
\tN$, $\mu^n_{k/n}=\nu^n_k$). This\vspace*{1pt} can be seen as a
discrete-in-time version (over a timestep of size $1/n$) of the
martingale property~(\ref{eqmartingale}) [which is actually a
characterization in law of a solution to~(\ref{deflmu}), as explained
below]. Thus, by sending $n$ to
infinity and assuming that a law of large number holds for the
empirical measure $\nu^n_{k}$, we expect $Y^{1,n}_t$ to converge to a solution
to~(\ref{edsnonlin}). The aim of Section~\ref{secproof} is to sketch
the rigorous proof of this result.

\subsection{Sketch of the rigorous proof}\label{secproof}

The next sections are, respectively, devoted to the three steps of the
proof of Theorem~\ref{chaos}. In Section~\ref{secmart}, we first
introduce a nonlinear martingale problem which is a weak formulation
of~(\ref{edsnonlin}): the law of any solution to this stochastic
differential equation solves the martingale problem. We check
uniqueness for the martingale problem by proving trajectorial
uniqueness for the stochastic differential equation (\ref{edsnonlin}).
Then, in
Section~\ref{sectight}, we check the tightness of the sequence of
laws of the processes $(Y^{1,n}_t)_{t\geq0}$. Because of the
exchangeability of the processes $((Y^{1,n}_t, \ldots, Y^{n,n}_t )_{t
\ge0})_{ n \ge1}$ and according\vspace*{1pt}
to~\cite{sznitman-91}, this is equivalent to the tightness of the sequence
$(\pi^n)_n$ of the laws of the empirical
measures
%
\begin{equation}
\label{eqmun} \mu^n=\frac{1}{n}\sum
_{i=1}^n\delta_{Y^{i,n}}
\end{equation}
considered
as random variables valued in the space ${\mathcal P}({\mathcal C})$
of probability measure on the set ${\mathcal C}$ of continuous paths
from $[0,+\infty)$ to $\R$. The space ${\mathcal C}$ is endowed with
the topology of uniform convergence on compact sets and ${\mathcal
P}({\mathcal C})$ with the corresponding topology for convergence in
distribution. The third and last step, performed in
Section~\ref{seclimit}, consists in checking that the limit $\pi
^\infty$
of any convergent subsequence of $(\pi^n)_n$ is concentrated on the
solutions of the martingale problem, which, in particular, provides
existence of a solution $P$ to this problem. A probability measure $Q$
on ${\mathcal C}$ with initial marginal $Q_0=m$ solves the martingale
problem if and only if $F(Q)=0$ for a countable set of functionals $F$
of the form (\ref{deff}) below. Since the chaoticity of the initial
conditions implies that $\pi^\infty(\{Q\in{\mathcal P}({\mathcal
C})\dvtx Q_0=m\})=1$, checking that $\E^{\pi_\infty}\llvert F(Q)\rrvert =0$ for all $F$
in this countable set is enough to conclude that $\pi^\infty=\delta
_P$. Combined with the results of the two first steps, this ensures
that the whole sequence $(\pi^n)_n$ converges weakly to $\delta_P$
where $P$ denotes the unique solution of the martingale problem, namely
the law of the unique solution to the stochastic differential equation
(\ref{edsnonlin}). According\vspace*{1pt} to \cite{sznitman-91}, this is
equivalent to the \mbox{$P$-}chaoticity of the processes $((Y^{1,n}_t,\ldots,Y^{n,n}_t)_{t\geq0})_{n\geq1}$ and this completes the proof of
Theorem~\ref{chaos}.

Finally, Section~\ref{secprop} is devoted to the proof of
Proposition~\ref{proplimiting-acceptance}.

As already mentioned, our main result combines a diffusion
approximation and a mean-field limit. Mean-field limits apply to
systems of $n$ interacting particles (here the components $Y^{i,n}$)
when the interaction between two particles is
of order~$1/n$. At first sight, it is not obvious that this is the case
for the system considered in the paper. Nevertheless, from the above
formal computation of $\E ( \varphi(Y^{1,n}_{(k+1)/n}) |
{\mathcal F}^n_k  )$, we see in equation (\ref{pdmdisc}) that
the interaction is actually of mean-field type: the other components
influence the evolution of $Y^{1,n}_{(k+1)/n}$ only through the
empirical measure $\mu^n_{k/n}=\frac{1}{n}\sum_{i=1}^n \delta
_{Y^{i,n}_{k/n}}$. The\vspace*{1pt} mean-field limit is a law of large numbers for
the empirical measure $\mu^n$ on the path-space: we prove that $\mu
^n$ converges to the unique solution $P$ of the martingale problem. In
the same time, we have to deal with the diffusion approximation.

Notice that in previous scaling results given in the literature, the
assumption that the vector of initial positions $(X^{1,n}_0,\ldots,X^{n,n}_0)$ is distributed according to the target density makes the
derivation of both the mean-field limit and the diffusion approximation
much easier: since at subsequent times, $(X^{1,n}_k,\ldots,X^{n,n}_k)$
remains distributed according to the target density, it is enough to
identify the limiting infinitesimal generator at the initial time.
Moreover, under this stationarity assumption and when the target
density is the $n$-fold product of a fixed probability density, the
mean-field limit is obtained by the standard law of large numbers.

We end this section with the following lemma which states some basic
properties of the functions $\Gamma$ and $\cG$.

\begin{lemma}\label{propggam}
The function $\Gamma$ is continuous on $[0,+\infty]\times\R$ and
such that
%
\begin{eqnarray}
\qquad && \inf_{(a,b)\in[0,+\infty]\times[\inf V'',\sup
V'']}\Gamma(a,b) > 0,\label{mingam}
\\
&& \exists C<+\infty, \forall(a,b)\mbox{ and }
\bigl(a',b' \bigr)\in[0,+\infty]\times\R,
\nonumber\\[-8pt]\label{lipgam} \\[-8pt]
&&\qquad \bigl\llvert \Gamma(a,b)-\Gamma \bigl(a',b'
\bigr)\bigr\rrvert \leq C \bigl(\bigl\llvert b'-b\bigr\rrvert +
\bigl\llvert a'-a\bigr\rrvert +\bigl\llvert \sqrt
{a'}- \sqrt{a}\bigr\rrvert \bigr).\nonumber
\end{eqnarray}
The function $\cG$ is continuous on $\{[0,+\infty]\times\R\}
\setminus\{(0,0)\}$ and such that
%
\begin{eqnarray}
&& \forall(a,b)\in[0,+\infty]\times\R,\qquad \sqrt{a}\cG(a,b)\leq \biggl(l^2
\sqrt{b^+}\vee\frac{2l}{\sqrt{2\pi}} \biggr),\label{majog}
\\
\label{lipg} && \exists C<+\infty, \forall(a,b)\mbox{ and }
\bigl(a',b' \bigr)\in [0,+\infty]\times \bigl[\inf
V'',\sup V'' \bigr],
\\
&&\qquad \bigl({\sqrt{a}}\wedge\sqrt{a'} \bigr)\bigl\llvert \cG(a,b)-\cG
\bigl(a',b' \bigr)\bigr\rrvert\nonumber
\\
&&\qquad\qquad \leq C \bigl(\bigl\llvert b'-b\bigr\rrvert +\bigl\llvert
a'-a\bigr\rrvert +\bigl\llvert \sqrt{a'}-\sqrt{a}
\bigr\rrvert \bigr).
\nonumber
\end{eqnarray}
Last,
%
\begin{equation}
\label{bornggam} \forall(a,b)\in[0,+\infty]\times\R,\qquad 0\leq{\mathcal G}(a,b)\leq
\Gamma(a,b)\leq l^2.
\end{equation}
\end{lemma}

Notice that $\cG$ is indeed discontinuous at point $(0,0)$ since
$\lim_{b \to0^+} \cG(0,b) \neq\cG(0,0)$. The proof of this lemma
is given in the \hyperref[secapp]{Appendix}.

\section{Uniqueness for the limiting diffusion}\label{secmart}
In the present section, we are going to prove trajectorial uniqueness
for the stochastic differential equation (\ref{edsnonlin}) nonlinear
in the sense of McKean and deduce uniqueness for the following weak
formulation of this dynamics.

%
\begin{definition}\label{defmp}
Let $(Y_t)_{t\geq0}$ denote the canonical process on ${\mathcal
C}$ and recall the definition (\ref{deflmu}) of $L_\mu$. A
probability measure $P\in{\mathcal P}({\mathcal C})$ with
time-marginals $(P_t)_{t\geq0}$ solves the nonlinear martingale
problem (MP) if $P_0=m$ and for any $\varphi\dvtx \R\to\R$ $C^2$ with
compact support,
\[
\biggl(M^\varphi_t\stackrel{\mathrm{def}}{=}
\varphi(Y_t)-\int_0^t
L_{P_s}\varphi(Y_s)\,ds \biggr)_{t\geq0}\qquad \mbox{is a
$P$-martingale}.
\]
\end{definition}

This martingale problem is the weak formulation of the nonlinear
stochastic differential equation (\ref{edsnonlin}). Indeed, the law of
any solution of (\ref{edsnonlin}) solves (MP). Conversely, when $P$
solves (MP), one easily checks by Paul L\'evy's characterization (see
\cite{karatzas-shreve-88}, Theorem 3.16, page~157) that
\[
\biggl(\beta_t=\int_0^t\frac{dY_s+{\cG}(\langle P_s,(V')^2\rangle,\langle
P_s,V''\rangle)V'(Y_s)\,ds}{\sqrt{\Gamma(\langle P_s,(V')^2\rangle,\langle
P_s,V''\rangle)}} \biggr)_{t\geq0}
\]
is a $P$-Brownian
motion. Thus, this implies the existence of a weak solution with law
$P$ for
the stochastic differential equation
\begin{eqnarray}\label{sde2}
X^P_t&=&\xi+\int_0^t
\Gamma^{1/2} \bigl( \bigl\langle P_s, \bigl(V'
\bigr)^2 \bigr\rangle, \bigl\langle P_s,V''
\bigr\rangle \bigr)\,dB_s
\nonumber\\[-8pt]\\[-8pt]\nonumber
&&{} -\int_0^t{
\cG} \bigl( \bigl\langle P_s, \bigl(V'
\bigr)^2 \bigr\rangle, \bigl\langle P_s,V''
\bigr\rangle \bigr)V' \bigl(X^P_s
\bigr)\,ds.\nonumber
\end{eqnarray}
For fixed time-dependent coefficients $\Gamma^{1/2}(\langle
P_s,(V')^2\rangle,\langle P_s,V''\rangle)$ and  ${\cG}(\langle
P_s,\break (V')^2\rangle,\langle P_s,V''\rangle)$, by boundedness of $\cG$
on $[0,+\infty]\times[\inf V'',\sup V'']$ (see
Lemma~\ref{propggam} above) and
Lipschitz continuity of $V'$, it is standard to check that trajectorial
uniqueness holds for this
(linear in the sense of McKean)
stochastic differential equation. As a consequence, by the Yamada--Watanabe theorem (see
\cite{karatzas-shreve-88}, Proposition 3.20, page~309, Corollary 3.23, page~310), this linear stochastic differential
equation admits a unique strong solution and the law of this solution
is~$P$. In conclusion, one may associate
a strong solution to~(\ref{edsnonlin}) with law $P$, to any solution
$P$ of the nonlinear martingale problem (MP).

Notice that the two next sections will ensure existence for (MP) and
(\ref{edsnonlin}).
Uniqueness is ensured by the following proposition.

\begin{proposition}For any probability measure $m$ on $\R$, 
uniqueness holds for the nonlinear martingale problem (MP) and
trajectorial uniqueness holds for the stochastic differential equation
(\ref{edsnonlin}).
\label{unicmp}
\end{proposition}

To prove Proposition~\ref{unicmp}, we need the following technical lemma.

\begin{lemma}\label{momv}
For any solution $(X_t)_{t\geq0}$ of (\ref{edsnonlin}),
\[
\forall0\leq s\leq t,\qquad \E \bigl[(X_t-X_s)^2
\bigr]\leq2l^2 \biggl[(t-s)+ \biggl(l^2\sup
\bigl(V'' \bigr)^+\vee\frac{2}{\pi} \biggr)
(t-s)^2 \biggr].
\]
Moreover, if $\langle m,(V')^2\rangle<+\infty$, then $t\mapsto\E
[(V'(X_t))^2]$ is locally bounded.
If $\langle m,(V')^2\rangle=+\infty$, then $\forall t\geq0$,
$\E[(V'(X_t))^2]=+\infty$.
\end{lemma}

\begin{pf}
Let $(X_t)_{t\geq0}$ solve (\ref{edsnonlin}). Then for $0\leq s\leq t$,
\begin{eqnarray*}
\E \bigl[(X_t-X_s)^2 \bigr]&\leq&2\E
\biggl[ \biggl(\int_s^t \Gamma^{1/2}
\bigl(\E \bigl[ \bigl(V'(X_r) \bigr)^2
\bigr],\E \bigl[V''(X_r) \bigr]
\bigr)\,dB_r \biggr)^2 \biggr]
\\
&&{}+2(t-s)\int_s^t\cG^2 \bigl(\E
\bigl[ \bigl(V'(X_r) \bigr)^2 \bigr],\E
\bigl[V''(X_r) \bigr] \bigr)\E \bigl[
\bigl(V'(X_r) \bigr)^2 \bigr]\,dr
\\
&\leq&2l^2(t-s)+2 \biggl(l^4\sup \bigl(V''
\bigr)^+\vee\frac{2l^2}{\pi} \biggr) (t-s)^2,
\end{eqnarray*}
where we used the boundedness properties of $\Gamma$ and $\sqrt{a}\cG
(a,b)$ stated in Lemma~\ref{propggam}.

One easily deduces the properties of $t\mapsto\E[(V'(X_t))^2]$ since
\begin{eqnarray*}
\bigl(V'(X_t) \bigr)^2&\geq&
\tfrac{1}{2} \bigl(V'(X_0) \bigr)^2-
\bigl(V'(X_t)-V'(X_0)
\bigr)^2
\\
&\geq&\tfrac{1}{2} \bigl(V'(\xi) \bigr)^2- \bigl
\llVert V''\bigr\rrVert _\infty^2(X_t-X_0)^2,
\\
\bigl(V'(X_t) \bigr)^2&\leq&2
\bigl(V'(X_0) \bigr)^2+2
\bigl(V'(X_t)-V'(X_0)
\bigr)^2
\\
&\leq& 2 \bigl(V'(\xi) \bigr)^2+2\bigl\llVert
V'' \bigr\rrVert _\infty^2(X_t-X_0)^2,
\end{eqnarray*}
with $\xi$ distributed according to $m$.
\end{pf}

We are now in position to prove Proposition~\ref{unicmp}.

\begin{pf*}{Proof of Proposition~\ref{unicmp}}
By the discussion following Definition~\ref{defmp}, we know that, for
a given Brownian motion $B_t$ and initial condition $\xi$, one may associate
a strong solution to (\ref{edsnonlin}) with law $P$ to any solution
$P$ of the nonlinear martingale problem (MP). Therefore, to get
uniqueness of solutions to (MP), it is enough
to prove trajectorial uniqueness for
(\ref{edsnonlin}).
Let $(X_t)_{t\geq0}$ and $(\tilde{X}_t)_{t\geq0}$ denote two
solutions of this nonlinear stochastic differential equation, with the
same initial condition, and driven by the same Brownian motion. If
$\langle m,(V')^2\rangle=+\infty$, then by Lemma~\ref{momv} and
since $\Gamma(\infty,b)=\frac{l^2}{2}$ and $\cG(\infty,b)=0$,
these two processes are equal to
$ (X_0+\frac{l B_t}{\sqrt{2}} )_{t\geq0}$. This proves
trajectorial
uniqueness.

Let us now assume that $\langle m,(V')^2\rangle<+\infty$. By
Lemma~\ref{momv},
$t\mapsto\E[(X_t-\tilde{X}_t)^2]=\E[(X_t-X_0-(\tilde{X}_t-\tilde
{X}_0))^2]$
and $t\mapsto\E[(V'(X_t))^2]\vee\E[(V'(\tilde{X}_t))^2]$ are locally
bounded. In order to simplify the notation, let us denote
\begin{eqnarray*}
\Gamma_s&=&\Gamma \bigl(\E \bigl[ \bigl(V'(X_s)
\bigr)^2 \bigr],\E \bigl[V''(X_s)
\bigr] \bigr),
\\
\tilde\Gamma_s&=&\Gamma \bigl(\E \bigl[
\bigl(V'(\tilde X_s) \bigr)^2 \bigr],\E
\bigl[V''(\tilde X_s) \bigr] \bigr)
\end{eqnarray*}
and
\begin{eqnarray*}
\cG_s &=& \cG \bigl(\E \bigl[ \bigl(V'(X_s)
\bigr)^2 \bigr],\E \bigl[V''(X_s)
\bigr] \bigr),
\\
\tilde \cG_s&=&\cG \bigl(\E \bigl[
\bigl(V'( \tilde X_s) \bigr)^2 \bigr],\E
\bigl[V''( \tilde X_s) \bigr] \bigr).
\end{eqnarray*}
Computing $(X_t-\tilde{X}_t)^2$ by It\^{o}'s formula and taking
expectations, one obtains
\begin{eqnarray}\label{eqx-tx}
\qquad \E \bigl[(X_t-\tilde{X}_t)^2 \bigr]\nonumber
&=& \int
_0^t \bigl(\Gamma^{1/2}_s-
\tilde \Gamma^{1/2}_s \bigr)^2\,ds
\nonumber\\[-8pt]\\[-8pt]
&&{} +2\E \biggl[\int_0^t \bigl(
\cG_s V'(X_s)-\tilde\cG_s
V'(\tilde{X}_s) \bigr) (\tilde{X}_s-X_s)\,ds
\biggr].\nonumber
\end{eqnarray}
One has, using (\ref{bornggam}) and the Cauchy--Schwarz inequality,
\begin{eqnarray*}
&&\E \bigl[ \bigl(\cG_s V'(X_s)-\tilde
\cG_s V'(\tilde{X}_s) \bigr) (
\tilde{X}_s-X_s) \bigr]
\\
&&\qquad =\cG_s \E \bigl[ \bigl( V'(X_s)-
V'(\tilde{X}_s) \bigr) (\tilde{X}_s-X_s)
\bigr] +(\cG_s -\tilde\cG_s) \E \bigl[
V'(\tilde X_s) (\tilde{X}_s-X_s)
\bigr]
\\
&&\qquad \le l^2 \bigl\llVert V''\bigr\rrVert
_\infty \E \bigl[ (X_s-\tilde{X}_s)^2
\bigr]+ \llvert \cG_s -\tilde\cG_s\rrvert
\E^{1/2} \bigl[ \bigl(V'(\tilde X_s)
\bigr)^2 \bigr] \E^{1/2} \bigl[ (\tilde {X}_s-X_s)^2
\bigr]
\end{eqnarray*}
which, combined with the similar inequality obtained by exchanging
$\tilde{X}$ and
$X$, yields
\begin{eqnarray*}
&& \E \bigl[ \bigl(\cG_s V'(X_s)-\tilde
\cG_s V'(\tilde{X}_s) \bigr) (
\tilde{X}_s-X_s) \bigr]
\\
&&\qquad \leq l^2 \bigl
\llVert V''\bigr\rrVert _\infty \E \bigl[
(X_s- \tilde{X}_s)^2 \bigr]
\\
&&\quad\qquad{}+\llvert \cG_s- \tilde\cG_s\rrvert \bigl(\E \bigl[
\bigl(V'({X}_s) \bigr)^2 \bigr]\wedge\E
\bigl[ \bigl(V'(\tilde{X}_s) \bigr)^2
\bigr] \bigr)^{1/2}\E ^{1/2} \bigl[(X_s-
\tilde{X}_s)^2 \bigr].
\end{eqnarray*}
Using this inequality to deal with the second term on the right-hand
side of (\ref{eqx-tx}) and (\ref{mingam}) to deal with the first one
then using the boundedness of $V''$ and (\ref{lipgam}), (\ref{lipg})
and Young's inequality, one obtains that
\begin{eqnarray*}
&& \E \bigl[(X_t-\tilde{X}_t)^2 \bigr]
\\
&&\qquad \leq
\frac{1}{4\inf_{a\geq
0,b\in[\inf V'',\sup
V'']}\Gamma(a,b)}\int_0^t (
\Gamma_s-\tilde \Gamma_s )^2\,ds
\\
&&\quad\qquad{}+2l^2\bigl\llVert V''\bigr\rrVert
_\infty\int_0^t \E
\bigl[(X_s- \tilde{X}_s)^2 \bigr]\,ds
\\
&&\quad\qquad{} +2\int_0^t\llvert \cG_s-
\tilde \cG_s\rrvert \bigl(\E \bigl[ \bigl(V'({X}_s)
\bigr)^2 \bigr]\wedge\E \bigl[ \bigl(V'(
\tilde{X}_s) \bigr)^2 \bigr] \bigr)^{1/2}
\E^{1/2} \bigl[(X_s-\tilde{X}_s)^2
\bigr]\,ds
\\
&&\qquad \leq C\int_0^t \E \bigl[(X_s-
\tilde{X}_s)^2 \bigr]+\E^2
\bigl[V''(X_s)-V''(
\tilde{X}_s) \bigr]
\\
&&\hspace*{22pt}\quad\qquad{}+\E^2 \bigl[ \bigl(V'(X_s)
\bigr)^2- \bigl(V'(\tilde {X}_s)
\bigr)^2 \bigr]
\\
&&\hspace*{22pt}\quad\qquad{} + \bigl(\E^{1/2} \bigl[ \bigl(V'(X_s)
\bigr)^2 \bigr]-\E^{1/2} \bigl[ \bigl(V'(
\tilde {X}_s) \bigr)^2 \bigr] \bigr)^2\,ds.
\end{eqnarray*}
Now, since
\begin{eqnarray*}
&& \bigl\llvert \E \bigl[V''(X_s)-V''(
\tilde{X}_s) \bigr]\bigr\rrvert \leq\bigl\llVert V^{(3)}
\bigr\rrVert _\infty \E ^{1/2} \bigl[(X_s-
\tilde{X}_s)^2 \bigr],
\\
&& \bigl\llvert \E \bigl[ \bigl(V'(X_s)
\bigr)^2- \bigl(V'(\tilde{X}_s)
\bigr)^2 \bigr]\bigr\rrvert
\\
&&\qquad \leq \bigl\llVert V''
\bigr\rrVert _\infty\E^{1/2} \bigl[(X_s-
\tilde{X}_s)^2 \bigr]
\bigl(\E^{1/2} \bigl[ \bigl(V'(X_s)
\bigr)^2 \bigr]+\E^{1/2} \bigl[ \bigl(V'(
\tilde{X}_s) \bigr)^2 \bigr] \bigr),
\\
&& \bigl\llvert \E^{1/2} \bigl[ \bigl(V'(X_s)
\bigr)^2 \bigr]-\E^{1/2} \bigl[ \bigl(V'(
\tilde{X}_s) \bigr)^2 \bigr]\bigr\rrvert
\\
&&\qquad \leq
\E^{1/2} \bigl[ \bigl(V'(X_s)-V'(
\tilde{X}_s) \bigr)^2 \bigr]
\\
&&\qquad \leq\bigl\llVert V''\bigr\rrVert _\infty
\E^{1/2} \bigl[(X_s-\tilde{X}_s)^2
\bigr],
\end{eqnarray*}
the local boundedness of $t\mapsto\E[(V'(X_t))^2]\vee\E[(V'(\tilde
{X}_t))^2]$, the local integrability of $t\mapsto\E[(X_t-\tilde
{X}_t)^2]$ and Gronwall's lemma ensure that $\forall t\geq0$, \mbox{$\E
[(X_t-\tilde{X}_t)^2]=0$}.
\end{pf*}

%
\begin{remark}
When $\langle m,(V')^2\rangle=+\infty$, we have already shown
uniqueness of
solutions to~(\ref{edsnonlin}), and it is actually easy to build a strong
solution. Indeed, since
\begin{eqnarray*}
\biggl(V' \biggl(\xi+\frac{l B_t}{\sqrt{2}} \biggr)
\biggr)^2&\geq& \frac{1}{2} \bigl(V'(\xi)
\bigr)^2-\frac{l^2\llVert  V''\rrVert  ^2_\infty B_t^2}{2},
\end{eqnarray*}
one has $\E[(V'(\xi+\frac{l B_t}{\sqrt{2}}))^2]=+\infty$ for all
$t\geq0$. As a consequence $ (\xi+\frac{l B_t}{\sqrt{2}}
)_{t\geq0}$ solves~(\ref{edsnonlin}).
\end{remark}

\section{Tightness}\label{sectight}

According to \cite{sznitman-91}, because of exchangeability, the
tightness of the sequence $(\pi^n)_n$ is equivalent to the tightness
of the laws of the processes $(Y^{1,n}_t)_{t\geq0}$. As a consequence,
the following proposition ensures that the sequence $(\pi^n)_n$ is
tight under the assumptions of Theorem~\ref{chaos}.

%
\begin{proposition}\label{proptightness}
Assume that the laws of the random variables $(X^{1,n}_0)_{n\geq1}$
are tight and that $\sup_n\E[(V'(X^{1,n}_0)^4)]<+\infty$. Then the
laws of the linearly interpolated processes $(Y^{1,n}_t=(\lceil
nt\rceil-nt)X^{1,n}_{\lfloor nt\rfloor}+(nt-\lfloor nt\rfloor
)X^{1,n}_{\lceil nt\rceil},t\geq0)_{n\geq1}$ are tight in $\mathcal C$.
Moreover,
%
\begin{equation}
\label{espvcar} t\mapsto\sup_{n\geq1}\E \bigl[
\bigl(V' \bigl(Y^{1,n}_t \bigr)
\bigr)^4 \bigr]\qquad\mbox{is locally bounded}.
\end{equation}
\end{proposition}

The proof of this proposition relies on the following estimate; the
proof of which is given after the one of the proposition.

\begin{lemma}\label{lemmom4}
Assume that $\sup_n\E[(V'(X^{1,n}_0))^4]<+\infty$. Then there exists
a finite constant $C$ depending on this supremum but not on $n$ such that
%
\begin{equation}
\label{eqmajaccroiss} \forall0\leq\underline{k}\leq\overline{k},\qquad \E \bigl(
\bigl(X^{1,n}_{{\overline{k}}}-X^{1,n}_{{\underline{k}}}
\bigr)^4 \bigr)\leq C \biggl(\frac{(\overline{k}-\underline{k})^2}{n^2}+e^{C({{\overline{k}}^4}/{n^4})}
\frac{(\overline{k}-\underline
{k})^4}{n^4} \biggr).\hspace*{-30pt}
\end{equation}
\end{lemma}

\begin{pf*}{Proof of Proposition~\ref{proptightness}}
Since the laws of the initial random variables $(X^{1,n}_0)_{n\geq1}$
are supposed to be tight, Kolmogorov criterion ensures the desired
tightness property as soon as there exists a nondecreasing function
$\gamma\dvtx \R_+\to\R_+$ such that
%
\begin{equation}
\label{majmomacc}\forall n\geq1, \forall0\leq s\leq t,\qquad \E \bigl(
\bigl(Y^{1,n}_t-Y^{1,n}_s
\bigr)^4 \bigr)\leq\gamma(t) (t-s)^2.
\end{equation}
Combining this estimation with the inequality
\[
\E \bigl[ \bigl(V' \bigl(Y^{1,n}_t \bigr)
\bigr)^4 \bigr]\leq 8\E \bigl[ \bigl(V'
\bigl(X^{1,n}_0 \bigr) \bigr)^4 \bigr]+8 \bigl
\llVert V''\bigr\rrVert ^4_\infty
\E \bigl[ \bigl(Y^{i,n}_t-Y^{i,n}_0
\bigr)^4 \bigr]
\]
one also easily checks that $t\mapsto\sup_{n\geq1}\E
[(V'(Y^{1,n}_t))^4]$ is locally bounded. Let us show how to deduce
(\ref{majmomacc}) from (\ref{eqmajaccroiss}).
For $t>s\geq0$ with $\lfloor nt\rfloor\geq\lceil ns\rceil$, using
(\ref{eqmajaccroiss}) for the second inequality, one obtains
\begin{eqnarray*}
&& \E \bigl( \bigl(Y^{1,n}_t-Y^{1,n}_s
\bigr)^4 \bigr)
\\
&&\qquad \leq 27 \E \biggl(\frac{(l(nt-\lfloor nt\rfloor)G^1_{\lceil
nt\rceil})^4}{n^2}+ \bigl(X^{1,n}_{\lfloor nt\rfloor}-X^{1,n}_{\lceil
ns\rceil}
\bigr)^4+\frac{(l(\lceil ns\rceil-ns)G^1_{\lceil ns\rceil
})^4}{n^2} \biggr)
\\
&&\qquad \leq\tilde C \biggl(\frac{(nt-\lfloor nt\rfloor)^2}{n^2}
\\
&&\hspace*{46pt}{}+ \biggl(\frac{(\lfloor nt\rfloor-\lceil ns\rceil)^2}{n^2}
+e^{Ct^4}
\frac{(\lfloor nt\rfloor-\lceil
ns\rceil)^4}{n^4} \biggr)
+\frac{(\lceil
ns\rceil-ns)^2}{n^2} \biggr)
\\
&&\qquad \leq C \bigl(1+t^2e^{Ct^4} \bigr) (t-s)^2.
\end{eqnarray*}
For $t>s\geq0$ with $\lfloor ns\rfloor=\lfloor nt\rfloor$, one has
$(nt-ns)^4\leq(nt-ns)^2$ and, therefore,
\[
\E \bigl( \bigl(Y^{1,n}_t-Y^{1,n}_s
\bigr)^4 \bigr)=\frac{l^4(nt-ns)^4}{n^2}\E \bigl( \bigl(G^1_{\lceil nt\rceil}
\bigr)^4 \bigr)\leq C(t-s)^2.
\]\upqed
\end{pf*}

The proof of Lemma~\ref{lemmom4} relies on the second inequality in
the next lemma, the proof of which is postponed to the \hyperref[secapp]{Appendix}.

\begin{lemma}\label{propdrift}
Let $x=(x_1,\ldots,x_n)\in\R^n$ and $\nu_n=\frac{1}{n}\sum_{i=1}^n\delta_{x_i}$. There exists a finite constant $C$ not
depending on $n$ and $x$ such that
%
\begin{eqnarray}
&& \E \bigl[ \bigl(e^{\sum_{i=1}^n (V(x_i)-V(x_i+({l}/{\sqrt{n}})G^i))}\wedge1\nonumber
\nonumber\\[-8pt]\label{reste}  \\[-8pt]
&&\hspace*{13pt}{} -e^{-\sum_{i=1}^n (({l}/{\sqrt{n}})V'(x_i)G^i+ ({l^2}/(2n))V''(x_i))}\wedge1
\bigr)^2 \bigr]
\leq \frac{C}{n},\nonumber
\\
&& \bigl\llvert \E \bigl(G^1 \bigl(1-e^{\sum_{i=1}^n (V(x_i)-V(x_i+({l}/{\sqrt{n}})G^i))} \bigr)^+ \bigr)
\bigr\rrvert \leq C \biggl(\frac
{\llvert V'(x_1)\rrvert }{\sqrt{n}}+\frac{1}{n}
\biggr),\label{restrev1}
\\
&& \biggl\llvert \E \bigl(G^1 \bigl(e^{\sum_{i=1}^n (V(x_i)-V(x_i+({l}/{\sqrt{n}})G^i))}\wedge1 \bigr)
\bigr)+ \frac{ V'(x_1)}{l\sqrt
{n}}{\cG} \bigl( \bigl\langle\nu_n,
\bigl(V' \bigr)^2 \bigr\rangle, \bigl\langle
\nu_n,V'' \bigr\rangle \bigr) \biggr
\rrvert\hspace*{-30pt}
\nonumber\\[-8pt]\label{restrev2} \\[-8pt]
&&\qquad  \leq C \biggl(\frac
{1+\llvert V'(x_1)\rrvert }{n}+\frac{\llvert V'(x_1)\rrvert }{n^{3/4}\langle\nu_n,(V')^2\rangle
^{1/4}}+\frac{\llvert V'(x_1)\rrvert ^{3/2}}{n^{3/4}\sqrt{\langle\nu
_n,(V')^2\rangle}}
\biggr).\nonumber
\end{eqnarray}
\end{lemma}

\begin{pf*}{Proof of Lemma~\ref{lemmom4}}
Let ${\overline{k}} >{\underline{k}}\geq0$. One has
%
\begin{eqnarray}\label{decompmom4}
&& \E \bigl( \bigl(X^{1,n}_{{\overline{k}}}-X^{1,n}_{{\underline{k}}}
\bigr)^4 \bigr)\nonumber
\\
&&\qquad \leq  \frac{8l^4}{n^2}\E \Biggl( \Biggl(\sum
_{k={\underline{k}}+1}^{{\overline{k}}}G_k^1
\Biggr)^4 \Biggr)+\frac{8l^4}{n^2}\E \Biggl( \Biggl(\sum
_{k={\underline{k}} +1}^{{\overline
{k}}}G^1_{k}1_{{\mathcal A}_k^c}
\Biggr)^4 \Biggr)
\nonumber\\[-8pt]\\[-8pt]
&&\qquad =\frac{24l^4({\overline{k}}-{\underline{k}})^2}{n^2}+\frac
{8l^4}{n^2}\sum_{{\underline{k}}+1\leq k_1,k_2,k_3,k_4\leq{\overline
{k}}}\E
\Biggl(\prod_{j=1}^4G^1_{k_j}1_{{\mathcal A}_{k_j}^c}
\Biggr)
\nonumber
\\
&&\qquad =\frac{24l^4({\overline{k}}-{\underline{k}})^2}{n^2}+\frac
{8l^4}{n^2} (T_{1,1,1,1}+T_{2,1,1}+T_{3,1}+T_{2,2}+T_4
),\nonumber
\end{eqnarray}
where the sum has been separated into five disjoint terms:
\begin{itemize}
\item$T_{1,1,1,1}$ corresponds to the restriction of the summation to
indexes $k_1$, $k_2$, $k_3$ and $k_4$ taking distinct values,
\item$T_{2,1,1}$ to the restriction to indexes such that the
cardinality of $\{k_1,k_2,k_3,k_4\}$ is equal to $3$,
\item$T_{3,1}$ to three indexes equal and the last one different,
\item$T_{2,2}$ to two pairs of equal indexes taking different values,
\item$T_4$ to four equal indexes.
\end{itemize}
One has%
\begin{eqnarray}\label{t4t22t31}
T_4+T_{2,2}+T_{3,1}
&\leq& ({\overline{k}}-{
\underline{k}})\E \bigl( \bigl(G^1_1 \bigr)^4
\bigr)+3({\overline{k}}-{\underline{k}}) ({\overline {k}}-{\underline{k}}-1)\E
\bigl( \bigl(G^1_1 \bigr)^2
\bigl(G^1_2 \bigr)^2 \bigr)\hspace*{-20pt}\nonumber
\\
&&{}+4({\overline{k}}-{\underline{k}}) ({\overline{k}}-{\underline {k}}-1)\E \bigl(
\bigl\llvert G^1_1\bigr\rrvert ^3 \bigr)\E
\bigl\llvert G^1_2\bigr\rrvert
\\
&=& 3({\overline {k}}-{\underline{k}})^2+\frac{16 ({\overline{k}}-{\underline
{k}})({\overline{k}}-{\underline{k}}-1)}{\pi}.\nonumber
\end{eqnarray}
Let us now estimate $T_{1,1,1,1}$ and $T_{2,1,1}$. For fixed $k_1$,
$k_2$, $k_3$ and $k_4$ (four integers in $\{\underline{k}+1, \ldots,\overline{k}\}$), let us define
$(\tilde{X}^{i,n}_{k},k\geq0)_{1\leq i\leq n}$ such that
$(\tilde{X}^{1,n}_{0},\ldots,\break\tilde
{X}^{n,n}_{0})=({X}^{1,n}_{0},\ldots, {X}^{n,n}_{0})$
and, for $k\geq0$ and $1\leq i\leq
n$,
\begin{eqnarray*}
\tilde{X}^{i,n}_{k+1}&=&\tilde{X}^{i,n}_k
+1_{\{k\notin\{
k_1-1,k_2-1,k_3-1,k_4-1\}\}}
\\
&&\hspace*{30pt}{}\times
\frac{l}{\sqrt{n}}G^i_{k+1}1_{ \{
U_{k+1}\leq
e^{\sum_{i=1}^n
 (V(\tilde{X}^{i,n}_k)-V  (\tilde{X}^{i,n}_k+({l}/{\sqrt{n}})G^i_{k+1} )
 )} \}}.
\end{eqnarray*}
Let us also denote by ${\mathcal F}$ the
sigma-field generated by these processes which are exchangeable,
independent of
$(U_k,(G^1_k,\ldots,G^n_k))_{k\in\{k_1,k_2,k_3,k_4\}}$ and equal to
the original processes $({X}^{i,n}_{k},k\geq1)_{1\leq i\leq n}$ on
the event
\[
\bigcap_{j=1}^4{\mathcal
A}_{k_j}^c=\bigcap_{j=1}^4
\bigl\{U_{k_j}>e^{\sum_{i=1}^n
 (V(X^{i,n}_{k_j-1})-V (X^{i,n}_{k_j-1}+({l}/{\sqrt{n}})G^i_{k_j}
 ) )} \bigr\}.
\]
When the indices $k_1,k_2,k_3$ and $k_4$ are distinct (namely for
$T_{1,1,1,1}$), by conditional independence of the vectors
$((G^1_{k_j},\ldots,G^n_{k_j},U_{k_j}))_{1\leq j\leq4}$ given
${\mathcal F}$, one has
\begin{eqnarray*}
&& \Biggl\llvert \E \Biggl(\prod_{j=1}^4G^1_{k_j}1_{{\mathcal
A}_{k_j}^c}
\Biggr) \Biggr\rrvert
\\
&&\qquad  = \Biggl\llvert \E \Biggl(\prod
_{j=1}^4G^1_{k_j}1_{ \{U_{k_j}>e^{\sum_{i=1}^n
 (V(\tilde{X}^{i,n}_{k_j-1})-V (\tilde
{X}^{i,n}_{k_j-1}+({l}/{\sqrt{n}})G^i_{k_j} )
 )} \}}\Biggr) \Biggr\rrvert
\\
&&\qquad = \Biggl\llvert \E \Biggl(\prod_{j=1}^4
\E \bigl(G^1_{k_j} \bigl(1-e^{\sum
_{i=1}^n (V(\tilde{X}^{i,n}_{k_j-1})-V(\tilde{X}^{i,n}_{k_j-1}+({l}/{\sqrt{n}})G^i_{k_j}))} \bigr)^+ |{
\mathcal F} \bigr) \Biggr) \Biggr\rrvert
\\
&&\qquad \leq\E \Biggl[\prod_{j=1}^4 \bigl
\llvert \E \bigl(G^1_{k_j} \bigl(1-e^{\sum_{i=1}^n
(V(\tilde{X}^{i,n}_{k_j-1})-V(\tilde{X}^{i,n}_{k_j-1}+({l}/{\sqrt{n}})G^i_{k_j}))} \bigr)^+
|{\mathcal F} \bigr) \bigr\rrvert \Biggr]
\\
&&\qquad \leq C\E \Biggl[\prod_{j=1}^4 \biggl(
\frac{1}{n}+\frac{\llvert V'(\tilde
{X}^{1,n}_{k_j-1})\rrvert }{\sqrt{n}} \biggr) \Biggr]
\\
&&\qquad \leq C \Biggl(
\frac
{1}{n^4}+\frac{1}{n^2}\E \Biggl[\sum_{j=1}^4
\bigl\llvert V' \bigl(\tilde {X}^{1,n}_{k_j-1}
\bigr)\bigr\rrvert ^4 \Biggr] \Biggr),
\end{eqnarray*}
where we used (\ref{restrev1}) for the last but one inequality and
Young's inequality for the last one.
Now for $k_1<k_2<k_3<k_4$, according to the above definition of
$(\tilde{X}^{i,n}_{k},k\geq0)_{1\leq i\leq n}$, the random vector
$(\tilde{X}^{1,n}_{k_j-1})_{1\leq j\leq4}$ has the same distribution
as $({X}^{1,n}_{k_j-j})_{1\leq j\leq4}$. Therefore,
\begin{eqnarray*}
T_{1,1,1,1}&\leq&4!C\sum_{\underline{k}+1\leq k_1<k_2<k_3<k_4\leq
\overline{k}} \Biggl(
\frac{1}{n^4}+\frac{1}{n^2}\E \Biggl[\sum_{j=1}^4
\bigl\llvert V' \bigl({X}^{1,n}_{k_j-j} \bigr)\bigr
\rrvert ^4 \Biggr] \Biggr)
\\
&=&4!C \Biggl(\frac{{\overline{k}-\underline{k}\choose 4}}{n^4}+\frac
{{\overline{k}-\underline{k}\choose 3}}{n^2}\sum
_{k=\underline
{k}}^{\overline{k}-4}\E \bigl[\bigl\llvert V'
\bigl({X}^{1,n}_{k} \bigr)\bigr\rrvert ^4 \bigr]
\Biggr).
\end{eqnarray*}
To deal with $T_{2,1,1}$ we remark that if, for instance, $k_2,k_3$ and
$k_4$ are distinct and $k_1=k_2$, then reasoning like above, and using
that $\E[(G^1_{k_1})^21_{{\mathcal
A}_{k_1}^c}\llvert {\mathcal F}]\leq\E[(G^1_{k_1})^2\rrvert {\mathcal F}]=1$, one obtains
\begin{eqnarray*}
&& \Biggl\llvert \E \Biggl(\prod_{j=1}^4G^1_{k_j}1_{{\mathcal
A}_{k_j}^c}\Biggr) \Biggr\rrvert
\\
&&\qquad \leq \E \Biggl[\prod_{j=3}^4
\bigl\llvert \E \bigl(G^1_{k_j} \bigl(1-e^{\sum_{i=1}^n
(V(\tilde{X}^{i,n}_{k_j-1})-V(\tilde{X}^{i,n}_{k_j-1}+({l}/{\sqrt{n}})G^i_{k_j}))}
\bigr)^+ |{\mathcal F} \bigr) \bigr\rrvert \Biggr]
\\
&&\qquad \leq C \Biggl(\frac{1}{n^2}+\frac{1}{n}\E \Biggl[\sum
_{j=3}^4\bigl\llvert V' \bigl(
\tilde{X}^{1,n}_{k_j-1} \bigr)\bigr\rrvert ^2 \Biggr]
\Biggr).
\end{eqnarray*}
One deduces that
\[
T_{2,1,1}\leq C\pmatrix{4
\cr
2} \Biggl(\frac{(\overline{k}-\underline
{k})(\overline{k}-\underline{k}-1)(\overline{k}-\underline
{k}-2)}{n^2}+
\frac{4{\overline{k}-\underline{k}\choose 2}}{n}\sum_{k=\underline{k}}^{\overline{k}-3}\E \bigl[
\bigl(V' \bigl({X}^{1,n}_{k} \bigr)
\bigr)^2 \bigr] \Biggr).
\]
By combining the estimations of $T_{3,1}+T_{2,2}+T_4$, $T_{1,1,1,1}$
and $T_{2,1,1}$ with Young's and Jensen's inequalities, one obtains that
%
\begin{eqnarray}\label{majac4}
&& \E \bigl( \bigl(X^{1,n}_{{\overline{k}}}-X^{1,n}_{{\underline
{k}}}\bigr)^4 \bigr)
\nonumber\\[-8pt]\\[-8pt]
&&\qquad \leq C \Biggl(\frac{(\overline{k}-\underline
{k})^2}{n^2}+\frac{(\overline{k}-\underline{k})^4}{n^6}+
\frac{(\overline{k}-\underline{k})^3}{n^4}\sum_{k=\underline
{k}}^{\overline{k}-1}\E \bigl[
\bigl(V' \bigl({X}^{1,n}_{k} \bigr)
\bigr)^4 \bigr] \Biggr).\nonumber
\end{eqnarray}
For the choice $\underline{k}=0$ and using $\sup_{n}\E
[(V'(X^{1,n}_0))^4]<+\infty$,%
\begin{equation}
\qquad\E \bigl[ \bigl(V' \bigl(X^{1,n}_k \bigr)
\bigr)^4 \bigr]\leq8\E \bigl[ \bigl(V'
\bigl(X^{1,n}_0 \bigr) \bigr)^4 \bigr]+8\bigl
\llVert V'' \bigr\rrVert _\infty^4
\E \bigl[ \bigl(X^{1,n}_k-X^{1,n}_0
\bigr)^4 \bigr]\label{eqvkvo},
\end{equation}
one obtains that
\[
\E \bigl( \bigl(X^{1,n}_{{\overline{k}}}-X^{1,n}_{0}
\bigr)^4 \bigr)\leq C \Biggl(\frac{{\overline{k}}^2}{n^2}+\frac{{\overline {k}}^4}{n^4}+
\frac{{\overline{k}}^3}{n^4}\sum_{k=0}^{\overline{k}-1}\E \bigl(
\bigl(X^{1,n}_{k}-X^{1,n}_{0}
\bigr)^4 \bigr) \Biggr).
\]
By a discrete version of Gronwall's lemma, one deduces that
\[
\forall{k}\geq0,\qquad \E \bigl( \bigl(X^{1,n}_{{{k}}}-X^{1,n}_{0}
\bigr)^4 \bigr)\leq Ce^{C({k^4}/{n^4})} \biggl(\frac{{{k}}^2}{n^2}\vee
\frac{{{k}}^4}{n^4} \biggr)\leq Ce^{C({k^4}/n^4)}.
\]
With (\ref{majac4}) and (\ref{eqvkvo}), one concludes that (\ref
{eqmajaccroiss}) holds.
\end{pf*}


\section{Identification of the limits of converging subsequences of \texorpdfstring{$(\pi^n)_{n\ge 1}$}{$(pi^n)_{n>= 1}$}}\label{seclimit}
From the previous section, we know that the sequence $(\pi^n)_n$ is tight.
Let $\pi^\infty$ denote the limit of a converging subsequence of
$(\pi^n)_n$ that we still index by $n$ for notational simplicity. We
want to prove that $\pi^\infty$ gives full weight to the solutions of
the nonlinear martingale problem (MP) (see Definition~\ref{defmp}).
To do so, for $\varphi\dvtx \R\to\R$ $C^3$ with compact support, $p\in
{\mathbb N}$, $g\dvtx \R^p\to\R$ continuous and bounded and $0\leq
s_1\leq s_2\leq\cdots\leq s_p\leq s\leq t$, we define
%
\begin{equation}\label{deff}
F\dvtx Q\in{\mathcal P}({\mathcal C})\mapsto \biggl\langle Q, \biggl(\varphi
(Y_t)-\varphi(Y_s)-\int_s^t
L_{Q_r}\varphi(Y_r)\,dr \biggr)g(Y_{s_1},
\ldots,Y_{s_p}) \biggr\rangle.\hspace*{-30pt}
\end{equation}
Since the chaoticity of the initial conditions implies that $\pi
^\infty(\{Q\in{\mathcal P}({\mathcal C})\dvtx Q_0=m\})=1$, to prove that
$\pi^\infty$ gives full weight to the solutions of
(MP), it is enough to check that $\E^{\pi^\infty}\llvert F(Q)\rrvert =0$. Indeed,
taking $g$ in a countable subset of the space of continuous functions
with compact support on $\R^p$ dense for the uniform convergence and
$(s_1,\ldots,s_p)$ in a countable dense subset of $[0,s]$, one obtains
\[
\pi^\infty \biggl( \biggl\{Q\in{\mathcal P}({\mathcal C})\dvtx
\E^Q \biggl(\varphi(Y_t)-\varphi(Y_s)-\int
_s^tL_{Q_r}\varphi(Y_r)\,dr
\Big|(Y_u)_{u\in[0,s]} \biggr)=0 \biggr\} \biggr)=1.
\]
Then taking $s,t$ in a countable dense subset of $\R_+$ and $\varphi$
in a countable subset of~$C^3$ functions with compact support on $\R$
dense in the space $C^2_c(\R)$ of $C^2$ functions with compact support
on $\R$ for the uniform convergence of the function and its
derivatives up to the order $2$, one concludes that
\[
\pi^\infty \biggl( \biggl\{Q\dvtx \forall\varphi\in C^2_c(
\R), \biggl(\varphi(Y_t)-\int_0^tL_{Q_r}
\varphi(Y_r)\,dr \biggr)_{t\geq0}\mbox{ is a $Q$-martingale}
\biggr\} \biggr)=1.
\]
In
Section~\ref{seclimit1}, we present the main steps of the
proof. Then, in Sections~\ref{seclimit2} and~\ref{seclimit3}, we
provide the proofs of the technical propositions stated and used in
Section~\ref{seclimit1}.

\subsection{Proof of \texorpdfstring{$\mathbb{E}^{\pi^\infty}|F(Q)|=0$}{$\mathbb{E}^{pi^{infty}}|F(Q)|=0$}}\label{seclimit1}
By combining the two next propositions, one first obtains the
asymptotic behavior of $\E^{\pi^n}\llvert F(Q)\rrvert =\E\llvert F(\mu^n)\rrvert $ as $n\to
\infty$.

%
\begin{proposition}\label{intmart}Let
\begin{eqnarray*}
M^{i,n}_k&=&\frac{l}{\sqrt{n}}\sum
_{j=0}^{k-1}\varphi '
\bigl(X^{i,n}_{j} \bigr) \bigl(G^i_{j+1}1_{{\mathcal
A}_{j+1}}-
\E \bigl[G^{i}_{j+1}1_{{\mathcal A}_{j+1}}|{\mathcal
F}^n_{j} \bigr] \bigr)
\\
&&{}+\frac{l^2}{2n}\sum_{j=0}^{k-1}\varphi
'' \bigl(X^{i,n}_{j} \bigr)
\bigl( \bigl(G^i_{j+1} \bigr)^21_{{\mathcal
A}_{j+1}}-
\E \bigl[ \bigl(G^{i}_{j+1} \bigr)^21_{{\mathcal A}_{j+1}}|{
\mathcal F}^n_{j} \bigr] \bigr).
\end{eqnarray*}
Under the assumptions of Theorem~\ref{chaos}, for all $s < t, \exists
C<\infty, \forall n\geq1$,
\[
\sup_{1\leq i\leq
n}\E \biggl\llvert \varphi \bigl(Y^{i,n}_t
\bigr)-\varphi \bigl(Y^{i,n}_s \bigr)-\int
_s^tL_{\mu
^n_r}\varphi
\bigl(Y^{i,n}_r \bigr)\,dr- \bigl(M^{i,n}_{\lceil
nt\rceil}-M^{i,n}_{\lceil ns\rceil}
\bigr) \biggr\rrvert \leq\frac{C}{n^{1/4}},
\]
where $\mu^n_r$
denotes the marginal at time $r$ of $\mu^n$ [defined by~(\ref{eqmun})].
\end{proposition}

%
\begin{proposition}\label{contmart}Under the assumptions of Theorem
\ref{chaos},
\[
\exists C<\infty, \forall n\geq1,\qquad \E \Biggl[ \Biggl(\frac{1}{n}\sum
_{i=1}^n \bigl(M^{i,n}_{\lceil nt\rceil}-M^{i,n}_{\lceil ns\rceil
}
\bigr)g \bigl(Y^{i,n}_{s_1},\ldots,Y^{i,n}_{s_p}
\bigr) \Biggr)^2 \Biggr]\leq\frac
{C}{\sqrt{n}}.
\]
\end{proposition}

Since
\[
F\bigl(\mu^n\bigr)=\frac{1}{n}\sum_{i=1}^n  \biggl(\varphi
\bigl(Y^{i,n}_t\bigr)-\varphi\bigl(Y^{i,n}_s\bigr)-\int_s^t L_{\mu^n_r}\varphi
\bigl(Y^{i,n}_r\bigr)\,dr \biggr)g\bigl(Y^{i,n}_{s_1},\ldots,Y^{i,n}_{s_p}\bigr),
\]
one has
\begin{eqnarray*}
&& \E\bigl\llvert F \bigl(\mu^n \bigr)\bigr\rrvert
\\
&&\qquad \leq
\frac{\llVert  g\rrVert  _\infty}{n}
\sum_{i=1}^n\E \biggl
\llvert \varphi \bigl(Y^{i,n}_t \bigr)-\varphi
\bigl(Y^{i,n}_s \bigr)-\int_s^tL_{\mu^n_r}
\varphi \bigl(Y^{i,n}_r \bigr)\,dr- \bigl(M^{i,n}_{\lceil nt\rceil}-M^{i,n}_{\lceil ns\rceil
}\bigr) \biggr\rrvert
\\
&&\quad\qquad{}+\E^{1/2} \Biggl[ \Biggl(\frac{1}{n}\sum
_{i=1}^n \bigl(M^{i,n}_{\lceil nt\rceil}-M^{i,n}_{\lceil ns\rceil}
\bigr)g \bigl(Y^{i,n}_{s_1},\ldots,Y^{i,n}_{s_p}
\bigr) \Biggr)^2 \Biggr].
\end{eqnarray*}
One deduces that
%
\begin{equation}
\label{limnul} \lim_{n\to\infty}\E^{\pi^n}\bigl\llvert F(Q)
\bigr\rrvert =0.
\end{equation}
Since $g$, $\cG$, $\Gamma$ and $V'\varphi'$ are bounded, the
function $F$ is bounded. Unfortunately, when $V'$ is not bounded, the
lack of continuity of $\mu\in{\mathcal P}(\R)\mapsto\langle\mu,(V')^2\rangle$ implies that $F$ is not continuous and the weak
convergence of $\pi^n$ to $\pi^\infty$ does not directly ensure that
$\E^{\pi^\infty}\llvert F(Q)\rrvert =0$.

To overcome this difficulty, for $k\in\tN$, we introduce the
second-order differential operator $L^k_\mu$ defined like $L_\mu$ in
(\ref{deflmu}) but with $\langle\mu,(V')^2\wedge k\rangle$
replacing $\langle\mu,(V')^2\rangle$. We also define $F_k$ like $F$
but with $L_{Q_r}$ replaced by $L^k_{Q_r}$. The functions $F_k$ are
uniformly bounded and converge pointwise to $F$ by the properties of
$\cG$ and $\Gamma$ stated in Lemma~\ref{propggam}. Moreover, $F_k$
is continuous. Indeed, to deal with the discontinuity of~$\cG$ at
$(0,0)$, it is enough to remark that for $\nu,\mu\in{\mathcal P}(\R)$,
\begin{eqnarray*}
&& \bigl\langle\nu,\bigl\llvert {\cG} \bigl( \bigl
\langle\nu,\bigl(V' \bigr)^2\wedge k \bigr
\rangle, \bigl\langle \nu,V'' \bigr\rangle \bigr)-{
\cG} \bigl( \bigl\langle\mu, \bigl(V' \bigr)^2\wedge k
\bigr\rangle, \bigl\langle \mu,V'' \bigr\rangle \bigr)
\bigr\rrvert \times\bigl\llvert V' \varphi'\bigr
\rrvert \bigr\rangle
\\
&&\qquad \leq 1_{\{\langle\mu,(V')^2\wedge
k\rangle>0\}}\bigl\llVert V' \varphi'\bigr
\rrVert _\infty
\\
&&\quad\qquad{}\times \bigl\llvert {\cG} \bigl( \bigl\langle\nu,
\bigl(V' \bigr)^2\wedge k \bigr\rangle, \bigl\langle
\nu,V'' \bigr\rangle \bigr)-{\cG} \bigl( \bigl\langle
\mu, \bigl(V' \bigr)^2\wedge k \bigr\rangle, \bigl
\langle \mu,V'' \bigr\rangle \bigr)\bigr\rrvert
\\
&&\quad\qquad{}+1_{\{\langle\mu,(V')^2\wedge
k\rangle=0\}} 2 l^2 \bigl\langle\nu-\mu,\bigl\llvert
V' \varphi'\bigr\rrvert \bigr\rangle,
\end{eqnarray*}
where we used in the last line the fact that $1_{\{\langle\mu,(V')^2\wedge
k\rangle=0\}} \langle\mu,\llvert V'\varphi'\rrvert \rangle=0$.
As a consequence,
\begin{eqnarray*}
\E^{\pi^\infty}\bigl\llvert F(Q)\bigr\rrvert &=&\lim_{k\to\infty}
\E^{\pi^\infty
}\bigl\llvert F_k(Q)\bigr\rrvert = \lim
_{k\to\infty}\lim_{n\to\infty}\E^{\pi^n}\bigl
\llvert F_k(Q)\bigr\rrvert
\\
&\leq& \limsup_{k\to\infty}\limsup_{n\to\infty}
\E^{\pi^n}\bigl\llvert F_k(Q)-F(Q)\bigr\rrvert,
\end{eqnarray*}
where we used (\ref{limnul}) for the inequality.
One concludes that $\E^{\pi^\infty}\llvert F(Q)\rrvert =0$ by the next proposition.

%
\begin{proposition}\label{propffk}Under the assumptions of Theorem
\ref{chaos},
\[
\lim_{k\to\infty}\sup_{n\geq1}\E\bigl\llvert
F_k \bigl(\mu^n \bigr)-F \bigl(\mu^n \bigr)
\bigr\rrvert =0.
\]
\end{proposition}

\subsection{Proof of Proposition~\texorpdfstring{\protect\ref{intmart}}{4}}\label{seclimit2}
This section is devoted to the proof of Proposition~\ref{intmart}. As
already pointed out in Section~\ref{secformal}, the
main difficulty is the identification of the drift term.

\begin{pf*}{Proof of Proposition~\ref{intmart}}
One has $dY^{i,n}_t=l \sqrt{n} G^i_{\lceil nt\rceil}1_{{\mathcal
A}_{\lceil nt\rceil}}\,dt$. As a
consequence,
\[
\varphi \bigl(Y^{i,n}_t \bigr)-\varphi
\bigl(Y^{i,n}_s \bigr)=\int_s^tl
\sqrt{n}\varphi ' \bigl(Y^{i,n}_r
\bigr)G^i_{\lceil
nr\rceil}1_{{\mathcal A}_{\lceil nr\rceil}}\,dr.
\]
Using the Taylor
expansion,
\begin{eqnarray*}
\varphi' \bigl(Y^{i,n}_r \bigr)&=&
\varphi' \bigl(X^{i,n}_{\lfloor
nr\rfloor} \bigr)+
\varphi'' \bigl(X^{i,n}_{\lfloor nr\rfloor}
\bigr) \bigl(nr-\lfloor nr\rfloor \bigr)\frac{l}{\sqrt{n}}G^i_{\lceil nr\rceil}1_{{\mathcal
A}_{\lceil nr\rceil}}
\\
&&{} +\varphi^{(3)} \bigl(\chi^{i,n}_{r} \bigr)
\bigl(nr-\lfloor nr\rfloor \bigr)^2\frac{l^2}{{2n}}
\bigl(G^i_{\lceil nr\rceil} \bigr)^21_{{\mathcal
A}_{\lceil nr\rceil}},
\end{eqnarray*}
with $\chi^{i,n}_{r}\in[X^{i,n}_{\lfloor nr\rfloor},Y^{i,n}_r]$, one
deduces that
\begin{eqnarray*}
&&\varphi \bigl(Y^{i,n}_t \bigr)-\varphi\bigl(Y^{i,n}_s \bigr)
\\
&&\quad{} -\int_s^t
\biggl( l\sqrt{n} \varphi' \bigl(X^{i,n}_{\lfloor nr\rfloor}
\bigr)G^i_{\lceil
nr\rceil}1_{{\mathcal A}_{\lceil
nr\rceil}}+ \frac{l^2}{2}
\varphi'' \bigl(X^{i,n}_{\lfloor
nr\rfloor}
\bigr) \bigl(G^i_{\lceil nr\rceil} \bigr)^21_{{\mathcal A}_{\lceil
nr\rceil}}
\biggr) \,dr
\\
&&\qquad =\frac{l^3}{2\sqrt{n}}\int_s^t
\varphi^{(3)} \bigl(\chi ^{i,n}_{r} \bigr) \bigl(nr-
\lfloor nr\rfloor \bigr)^2 \bigl(G^i_{\lceil nr\rceil}
\bigr)^31_{{\mathcal A}_{\lceil
nr\rceil}}\,dr
\\
&&\quad\qquad{}+\frac{l^2(ns-\lfloor ns\rfloor)(\lceil ns\rceil
-ns)}{2n}\varphi'' \bigl(X^{i,n}_{\lfloor ns\rfloor}
\bigr) \bigl(G^i_{\lceil
ns\rceil} \bigr)^21_{{\mathcal A}_{\lceil ns\rceil}}
\\
&&\quad\qquad{} -\frac{l^2(nt-\lfloor nt\rfloor)(\lceil nt\rceil
-nt)}{2n}\varphi''
\bigl(X^{i,n}_{\lfloor nt\rfloor} \bigr) \bigl(G^i_{\lceil nt\rceil
}
\bigr)^21_{{\mathcal A}_{\lceil nt\rceil}}.
\end{eqnarray*}
By the boundedness of $\varphi''$ and $\varphi^{(3)}$, one easily
concludes that
%
\begin{eqnarray}\label{decompin}
&& \E \biggl\llvert \varphi \bigl(Y^{i,n}_t
\bigr)- \varphi \bigl(Y^{i,n}_s \bigr)
-\int _s^tl \sqrt {n}\varphi'
\bigl(X^{i,n}_{\lfloor nr\rfloor} \bigr)G^i_{\lceil nr\rceil
}1_{{\mathcal A}_{\lceil nr\rceil}}
\nonumber\\[-8pt]\\[-8pt]\nonumber
&&\hspace*{83pt}{} +
\frac{l^2}{2}\varphi '' \bigl(X^{i,n}_{\lfloor nr\rfloor}
\bigr) \bigl(G^i_{\lceil nr\rceil} \bigr)^21_{{\mathcal
A}_{\lceil nr\rceil}}\,dr
\biggr\rrvert \leq\frac{C}{\sqrt{n}}.
\end{eqnarray}
%
To complete the proof, we now consider the decomposition
%
\begin{eqnarray}\label{decomper}
&& \int_s^t l\sqrt{n}\varphi'
\bigl(X^{i,n}_{\lfloor nr\rfloor} \bigr)G^i_{\lceil
nr\rceil}1_{{\mathcal A}_{\lceil nr\rceil}}+
\frac{l^2}{2}\varphi '' \bigl(X^{i,n}_{\lfloor nr\rfloor}
\bigr) \bigl(G^i_{\lceil nr\rceil} \bigr)^21_{{\mathcal
A}_{\lceil nr\rceil}}\,dr\nonumber
\\
&&\quad{} -
\int_s^tL_{\mu^n_r}\varphi
\bigl(Y^{i,n}_r \bigr)\,dr- \bigl(M^{i,n}_{\lceil nt\rceil}-M^{i,n}_{\lceil ns\rceil
}\bigr)
\\
&&\qquad =T^{i,n}_1+T^{i,n}_2+T^{i,n}_3-T^{i,n}_4+T^{i,n}_5,\nonumber
\end{eqnarray}
where
\begin{eqnarray*}
T^{i,n}_1&=&\int_s^t
\varphi' \bigl(X^{i,n}_{\lfloor
nr\rfloor} \bigr) \bigl(l
\sqrt{n} \E \bigl[G^{i}_{\lceil nr\rceil}1_{{\mathcal
A}_{\lceil nr\rceil}}|{\mathcal
F}^n_{\lfloor
nr\rfloor} \bigr]
\\
&&\hspace*{59pt}{}+\cG \bigl( \bigl\langle
\mu^n_{{\lfloor
nr\rfloor}/{n}}, \bigl(V' \bigr)^2
\bigr\rangle, \bigl\langle\mu^n_{{{\lfloor
nr\rfloor}/{n}}},V''
\bigr\rangle \bigr)V' \bigl(X^{i,n}_{\lfloor
nr\rfloor} \bigr)
\bigr)\,dr,
\\
T^{i,n}_2&=&\frac{1}{2}\int_s^t
\varphi'' \bigl(X^{i,n}_{\lfloor
nr\rfloor}
\bigr) \bigl(l^2\E \bigl[ \bigl(G^{i}_{\lceil nr\rceil}
\bigr)^21_{{\mathcal
A}_{\lceil nr\rceil}}|{\mathcal F}^n_{\lfloor
nr\rfloor}
 \bigr]
 \\
 &&\hspace*{70pt}{}-\Gamma \bigl( \bigl\langle\mu^n_{{{\lfloor
nr\rfloor}/{n}}},
\bigl(V' \bigr)^2 \bigr\rangle, \bigl\langle
\mu^n_{{{\lfloor
nr\rfloor}/{n}}},V'' \bigr\rangle
\bigr)\bigr)\,dr,
\\
T^{i,n}_3&=&\int_s^tL_{\mu^n_{{\lfloor
nr\rfloor}/{n}}}
\varphi \bigl(Y^{i,n}_{{\lfloor
nr\rfloor}/{n}} \bigr)-L_{\mu^n_r}\varphi
\bigl(Y^{i,n}_r \bigr)\,dr,
\\
T^{i,n}_4&=& \biggl(\frac{l(\lceil
nt\rceil-nt)}{\sqrt{n}}\varphi'
\bigl(X^{i,n}_{\lfloor
nt\rfloor} \bigr) \bigl(G^i_{\lceil nt\rceil}1_{{\mathcal A}_{\lceil
nt\rceil}}-
\E \bigl[G^{i}_{\lceil nt\rceil}1_{{\mathcal A}_{\lceil
nt\rceil}}|{\mathcal
F}^n_{\lfloor nt\rfloor} \bigr] \bigr)
\\
&&\hspace*{6pt}{}+\frac{l^2(\lceil
nt\rceil-nt)}{2n}\varphi''
\bigl(X^{i,n}_{\lfloor nt\rfloor} \bigr) \bigl( \bigl(G^i_{\lceil
nt\rceil}
\bigr)^21_{{\mathcal A}_{\lceil nt\rceil}}-\E \bigl[ \bigl(G^{i}_{\lceil
nt\rceil}
\bigr)^21_{{\mathcal A}_{\lceil nt\rceil}}|{\mathcal F}^n_{\lfloor nt\rfloor}
\bigr] \bigr)\biggr)
\end{eqnarray*}
and
\begin{eqnarray*}
T^{i,n}_5 &=& \biggl(\frac{l(\lceil
ns\rceil-ns)}{\sqrt{n}}\varphi'
\bigl(X^{i,n}_{\lfloor
ns\rfloor} \bigr) \bigl(G^i_{\lceil ns\rceil}1_{{\mathcal A}_{\lceil
ns\rceil}}-
\E \bigl[G^{i}_{\lceil ns\rceil}1_{{\mathcal A}_{\lceil
ns\rceil}}|{\mathcal
F}^n_{\lfloor
ns\rfloor} \bigr] \bigr)
\\
&&\hspace*{6pt}{}+\frac{l^2(\lceil
ns\rceil-ns)}{2n}\varphi''
\bigl(X^{i,n}_{\lfloor ns\rfloor} \bigr) \bigl( \bigl(G^i_{\lceil
ns\rceil}
\bigr)^21_{{\mathcal A}_{\lceil ns\rceil}}-\E \bigl[ \bigl(G^{i}_{\lceil
ns\rceil}
\bigr)^21_{{\mathcal A}_{\lceil ns\rceil}}|{\mathcal F}^n_{\lfloor ns\rfloor}
\bigr] \bigr)\biggr).
\end{eqnarray*}
The boundedness of $\varphi'$ and $\varphi''$ implies that
%
\begin{equation}
\E \bigl(\bigl\llvert T^{i,n}_4\bigr\rrvert +\bigl\llvert
T^{i,n}_5\bigr\rrvert \bigr)\leq\frac{C}{\sqrt n}\label{t4t5}.
\end{equation}
By (\ref{restrev2}), H\"{o}lder's inequality and the equality%
\begin{equation}
\E \biggl[\frac{(V'(Y^{i,n}_{\lfloor nr\rfloor/n}))^2}{\langle\mu
^n_{\lfloor nr\rfloor/n},(V')^2\rangle} \biggr]=1\label{consexch}
\end{equation}
deduced from exchangeability, one obtains
%
\begin{eqnarray}\label{t1}
\E\bigl\llvert T^{i,n}_1\bigr\rrvert &\leq& C\int
_s^t \frac{1+\E\llvert V'(Y^{i,n}_{\lfloor
nr\rfloor/n})\rrvert }{\sqrt{n}}+\frac{1}{n^{1/4}}\E
\biggl\llvert \frac
{V'(Y^{i,n}_{\lfloor nr\rfloor/n})}{\langle\mu^n_{\lfloor nr\rfloor
/n},(V')^2\rangle^{1/4}} \biggr\rrvert \nonumber
\\
&&\hspace*{23pt}{} +\frac{1}{n^{1/4}}\E \biggl(
\frac{\llvert V'(Y^{i,n}_{\lfloor nr\rfloor/n})\rrvert ^{3/2}}{\langle\mu^n_{\lfloor
nr\rfloor/n},(V')^2\rangle^{1/2}} \biggr)\,dr
\nonumber\\[-8pt]\\[-8pt]\nonumber
&\leq& C\int_s^t\frac{1+\E\llvert V'(Y^{i,n}_{\lfloor nr\rfloor
/n})\rrvert }{\sqrt{n}}+
\frac{\E^{3/4}(\llvert V'(Y^{i,n}_{\lfloor nr\rfloor
/n})\rrvert ^{2/3})}{n^{1/4}}
\\
&&\hspace*{23pt}{} +\frac{\E^{1/2}\llvert V'(Y^{i,n}_{\lfloor nr\rfloor/n})\rrvert }{n^{1/4}}\,dr.\nonumber
\end{eqnarray}
Concerning~$T^{i,n}_2$, by Cauchy--Schwarz inequality and (\ref
{reste}), one easily checks that
\begin{eqnarray*}
&& \bigl\llvert \E \bigl[ \bigl(G^{i}_{\lceil nr\rceil}
\bigr)^21_{{\mathcal A}_{\lceil
nr\rceil}}| {\mathcal F}^n_{\lfloor
nr\rfloor}
\bigr]
\\
&&\hspace*{3pt}{}-\E \bigl[ \bigl(G^{i}_{\lceil
nr\rceil} \bigr)^2
\bigl(e^{-\sum_{l=1}^n(V'(X^{l,n}_{\lfloor
nr\rfloor}) ({l}/{\sqrt{n}})G^{l}_{\lceil
nr\rceil}+({l^2}/(2n))V''(X^{l,n}_{\lfloor
nr\rfloor}))}\wedge1 \bigr)| {\mathcal F}^n_{\lfloor
nr\rfloor}
\bigr] \bigr\rrvert
\\
&&\qquad \leq\frac{C}{\sqrt{n}}.
\end{eqnarray*}
Moreover, by (\ref{calesp3}) and (\ref{calesp4}),
%
\begin{eqnarray}\label{espcondprod}
&& \E \bigl[G^{i}_{\lceil nr\rceil}G^{j}_{\lceil
nr\rceil}
\bigl(e^{-\sum_{l=1}^n(V'(X^{l,n}_{\lfloor
nr\rfloor}) ({l}/{\sqrt{n}})G^{l}_{\lceil
nr\rceil}+({l^2}/(2n))V''(X^{l,n}_{\lfloor
nr\rfloor}))}\wedge1 \bigr)|{\mathcal F}^n_{\lfloor
nr\rfloor}
\bigr]\nonumber
\\
&&\qquad =\frac{1_{\{i=j\}}}{l^2}\Gamma \bigl( \bigl\langle \mu^n_{{\lfloor nr\rfloor}/{n}},
\bigl(V' \bigr)^2 \bigr\rangle, \bigl\langle
\mu^n_{{\lfloor nr\rfloor}/{n}},V'' \bigr\rangle
\bigr)
\nonumber\\[-8pt]\\[-8pt]
&&\quad\qquad{}+\frac{V'(X^{i,n}_{\lfloor nr\rfloor})V'(X^{j,n}_{\lfloor nr\rfloor
})}{n}
\biggl({\mathcal G} \bigl( \bigl\langle
\mu^n_{{\lfloor nr\rfloor
}/{n}}, \bigl(V' \bigr)^2
\bigr\rangle, \bigl\langle \mu^n_{{\lfloor nr\rfloor
}/{n}},V''
\bigr\rangle \bigr)\nonumber
\\
&&\hspace*{142pt}{}-\frac{l^2e^{-((({l}/{2})\langle\mu
^n_{\lfloor nr\rfloor/n},V''\rangle)^2)/(2\langle\mu^n_{\lfloor
nr\rfloor/n},(V')^2\rangle)}}{\sqrt{2\pi l^2\langle\mu^n_{\lfloor
nr\rfloor/n},(V')^2\rangle}} \biggr).\hspace*{-20pt}\nonumber
\end{eqnarray}
(We will need this expression for $i \neq j$ below.) With the
boundedness of $\cG$ and~(\ref{consexch}), this implies that
%
\begin{equation}\label{t2}
\qquad \E\bigl\llvert T^{i,n}_2\bigr\rrvert 
\leq\frac{C}{\sqrt{n}}+\frac{C}{n}\int_s^t
\E \bigl[ \bigl(V' \bigl(Y^{i,n}_{{\lfloor nr\rfloor}/{n}} \bigr)
\bigr)^2 \bigr]+\E^{1/2} \bigl[ \bigl(V'
\bigl(Y^{i,n}_{{\lfloor
nr\rfloor}/{n}} \bigr) \bigr)^2
\bigr]\,dr.
\end{equation}

To deal with $T^{i,n}_3$, one remarks that by exchangeability,
boundedness of $\cG$, $\varphi'$ and $(V'\varphi')'$, then by (\ref{lipg})
\begin{eqnarray*}
&& \E \bigl\llvert \cG \bigl( \bigl\langle\mu^n_r,
\bigl(V' \bigr)^2 \bigr\rangle, \bigl\langle\mu
^n_r,V'' \bigr\rangle
\bigr)V'\varphi' \bigl(Y^{i,n}_r
\bigr)
\\
&&\hspace*{8pt}{}-\cG \bigl( \bigl\langle\mu^n_{\lfloor
nr\rfloor/n},
\bigl(V' \bigr)^2 \bigr\rangle, \bigl\langle
\mu^n_{\lfloor nr\rfloor
/n},V'' \bigr\rangle
\bigr)V'\varphi' \bigl(Y^{i,n}_{\lfloor nr\rfloor/n}
\bigr) \bigr\rrvert
\\
&&\qquad \leq\E \bigl(\bigl\llvert \cG \bigl( \bigl\langle\mu^n_r,
\bigl(V' \bigr)^2 \bigr\rangle, \bigl\langle
\mu^n_r,V'' \bigr\rangle
\bigr)
\\
&&\hspace*{45pt}{}-\cG \bigl( \bigl\langle\mu^n_{\lfloor
nr\rfloor/n},
\bigl(V' \bigr)^2 \bigr\rangle, \bigl\langle
\mu^n_{\lfloor
nr\rfloor/n},V'' \bigr\rangle
\bigr)\bigr\rrvert
\bigl( \bigl\langle\mu^n_r,\bigl\llvert
V' \varphi'\bigr\rrvert \bigr\rangle\wedge \bigl
\langle\mu^n_{\lfloor nr\rfloor/n},\bigl\llvert V'
\varphi'\bigr\rrvert \bigr\rangle \bigr) \bigr)
\\
&&\quad\qquad{} +C\E\bigl\llvert
Y^{i,n}_r-Y^{i,n}_{\lfloor nr\rfloor/n}\bigr\rrvert
\\
&&\qquad \leq C\E \bigl(\bigl\llvert \bigl\langle\mu^n_r-
\mu^n_{\lfloor
nr\rfloor/n},V'' \bigr\rangle
\bigr\rrvert +\bigl\llvert \bigl\langle\mu^n_r-
\mu^n_{\lfloor
nr\rfloor/n}, \bigl(V' \bigr)^2
\bigr\rangle\bigr\rrvert
\\
&&\hspace*{51pt}{} +\bigl\llvert \bigl\langle\mu^n_r-
\mu^n_{\lfloor nr\rfloor/n}, \bigl(V' \bigr)^2
\bigr\rangle \bigr\rrvert ^{1/2}+\bigl\llvert Y^{i,n}_r-Y^{i,n}_{\lfloor nr\rfloor/n}
\bigr\rrvert \bigr).
\end{eqnarray*}
By exchangeability, $\E\llvert \langle\mu^n_r-\mu^n_{\lfloor nr\rfloor
/n},V''\rangle\rrvert \leq\llVert  V^{(3)}\rrVert  _\infty\E\llvert Y^{i,n}_r-Y^{i,n}_{\lfloor
nr\rfloor/n}\rrvert $. Moreover, $\llvert Y^{i,n}_r-Y^{i,n}_{\lfloor nr\rfloor
/n}\rrvert \leq\frac{l}{\sqrt{n}}\llvert G^i_{\lceil nr\rceil}\rrvert $.
Dealing in the same way with the diffusion term by boundedness of
$\Gamma$ and $\varphi^{(3)}$ and (\ref{lipgam}), one deduces that
%
\begin{eqnarray}\label{t3}
\E\bigl\llvert T^{i,n}_3\bigr\rrvert &\leq&
\frac{C}{\sqrt{n}}+\int_s^t\E\bigl\llvert
\bigl\langle\mu ^n_r- \mu^n_{\lfloor nr\rfloor/n},
\bigl(V' \bigr)^2 \bigr\rangle\bigr\rrvert
\nonumber\\[-8pt]\\[-8pt]\nonumber
&&\hspace*{43pt}{} +
\E^{1/2}\bigl\llvert \bigl\langle \mu^n_r-
\mu^n_{\lfloor nr\rfloor/n}, \bigl(V' \bigr)^2
\bigr\rangle\bigr\rrvert \,dr.
\end{eqnarray}
One has
%
\begin{eqnarray} \label{majoecdisc}
&& \E\bigl\llvert \bigl\langle\mu^n_r-
\mu^n_{\lfloor
nr\rfloor/n}, \bigl(V' \bigr)^2
\bigr\rangle\bigr\rrvert
\nonumber
\\
&&\qquad \leq \sqrt{2}\bigl\llVert V''\bigr\rrVert
_\infty\E ^{1/2} \bigl[ \bigl(V'
\bigl(Y^{i,n}_r \bigr) \bigr)^2+
\bigl(V' \bigl(Y^{i,n}_{\lfloor
nr\rfloor/n} \bigr)
\bigr)^2 \bigr]
\nonumber\\[-8pt]\\[-8pt]\nonumber
&&\quad\qquad{}\times \E^{1/2} \bigl[ \bigl(Y^{i,n}_r-Y^{i,n}_{\lfloor
nr\rfloor/n}
\bigr)^2 \bigr]
\\
&&\qquad \leq\frac{C}{\sqrt{n}}\E ^{1/2} \bigl[ \bigl(V'
\bigl(Y^{i,n}_r \bigr) \bigr)^2+
\bigl(V' \bigl(Y^{i,n}_{\lfloor nr\rfloor/n} \bigr)
\bigr)^2 \bigr].\nonumber
\end{eqnarray}
Plugging this inequality in (\ref{t3}) and inserting the resulting
inequality together with (\ref{t4t5}), (\ref{t1}) and (\ref{t2})
into (\ref{decomper}), one concludes with (\ref{decompin}) and the
local boundedness of $r\mapsto\sup_{n\geq1}\sup_{1\leq i\leq n}\E
[(V'(Y^{i,n}_r))^2]$ deduced from (\ref{espvcar}) and exchangeability.
\end{pf*}

This completes the proof of Proposition~\ref{intmart}.

\subsection{Proofs of Propositions \texorpdfstring{\protect\ref{contmart}}{5} and \texorpdfstring{\protect\ref{propffk}}{6}}\label{seclimit3}
Finally, it remains to prove Propositions~\ref{contmart} and~\ref{propffk}.

\begin{pf*}{Proof of Proposition~\ref{contmart}}
Since for $1\leq i\leq n$, $(M^{i,n}_k)$ is a ${\mathcal
F}^n_k$-martingale and $g(Y^{i,n}_{s_1},\ldots,Y^{i,n}_{s_p})$ is
${\mathcal F}^n_{\lceil ns\rceil}$-measurable, one has
%
\begin{eqnarray}\label{sumespcond}
&& \E \Biggl[ \Biggl(\frac{1}{n}\sum_{i=1}^n
\bigl(M^{i,n}_{\lceil nt\rceil
}-M^{i,n}_{\lceil ns\rceil} \bigr)g
\bigl(Y^{i,n}_{s_1},\ldots,Y^{i,n}_{s_p}
\bigr) \Biggr)^2 \Biggr]
\nonumber
\\
&&\qquad =\frac{1}{n^2}\sum_{i,j=1}^n \sum
_{k=\lceil ns\rceil}^{\lceil
nt\rceil-1}\E \bigl[\E \bigl[
\bigl(M^{i,n}_{k+1}-M^{i,n}_k \bigr)
\bigl(M^{j,n}_{k+1}-M^{j,n}_k \bigr)|{
\mathcal F}^n_k \bigr]
\\
&&\hspace*{113pt}{}\times g \bigl(Y^{i,n}_{s_1},
\ldots,Y^{i,n}_{s_p} \bigr)g \bigl(Y^{j,n}_{s_1},
\ldots,Y^{j,n}_{s_p} \bigr) \bigr].\nonumber
\end{eqnarray}
Using the boundedness of $\varphi'$ and $\varphi''$, then (\ref
{reste}), (\ref{espcondprod}) and the equality
\begin{eqnarray*}
&& \E \bigl[G^{i}_{k+1} \bigl(e^{-\sum_{l=1}^n(V'(X^{l,n}_{k})
({l}/{\sqrt{n}})G^{l}_{k+1}+({l^2}/(2n))V''(X^{l,n}_{k}))}\wedge 1
\bigr)|{\mathcal F}^n_k \bigr]
\\
&&\qquad =-\frac{ V'(X^{i,n}_{k})}{l\sqrt
{n}}\cG
\bigl( \bigl\langle\mu^n_{k/n}, \bigl(V'
\bigr)^2 \bigr\rangle, \bigl\langle\mu ^n_{k/n},V''
\bigr\rangle \bigr)
\end{eqnarray*}
deduced from (\ref{calesp1}), one obtains
\begin{eqnarray*}
&& \bigl\llvert \E \bigl[ \bigl(M^{i,n}_{k+1}-M^{i,n}_k
\bigr) \bigl(M^{j,n}_{k+1}-M^{j,n}_k
\bigr)|{\mathcal F}^n_k \bigr] \bigr\rrvert
\\
&&\qquad \leq\frac{C}{n} \bigl\llvert \E \bigl[G^{i}_{k+1}G^{j}_{k+1}1_{{\mathcal
A}_{k+1}}|{
\mathcal F}^n_k \bigr]-\E \bigl[G^{i}_{k+1}1_{{\mathcal
A}_{k+1}}|{
\mathcal F}^n_k \bigr]\E \bigl[G^{j}_{k+1}1_{{\mathcal
A}_{k+1}}|{
\mathcal F}^n_k \bigr] \bigr\rrvert
\\
&&\qquad \leq\frac{C}{n^{3/2}}
\\
&&\quad\qquad{} +\frac{C}{n} \bigg|\E \bigl[G^{i}_{k+1}G^{j}_{k+1}
\bigl(e^{-\sum_{l=1}^n(V'(X^{l,n}_{k})
({l}/{\sqrt{n}})G^{l}_{k+1}+({l^2}/(2n))V''(X^{l,n}_{k}))}\wedge 1 \bigr)|{\mathcal F}^n_k
\bigr]
\\
&&\hspace*{24pt}\quad\qquad{}-\E \bigl[G^{i}_{k+1} \bigl(e^{-\sum_{l=1}^n(V'(X^{l,n}_{k}) ({l}/{\sqrt
{n}})G^{l}_{k+1}+({l^2}/(2n))V''(X^{l,n}_{k}))} \wedge1
\bigr)|{\mathcal F}^n_k \bigr]
\\
&&\hspace*{38pt}\quad\qquad{}\times\E \bigl[G^{j}_{k+1} \bigl(e^{-\sum
_{l=1}^n(V'(X^{l,n}_{k})({l}/{\sqrt{n}})G^{l}_{k+1}+
({l^2}/(2n))V''(X^{l,n}_{k}))}
\wedge1 \bigr)|{\mathcal F}^n_k \bigr] \bigg|
\\
&&\qquad \leq C \biggl(\frac{1}{n^{3/2}}+\frac{1_{\{i=j\}}}{n}+\frac
{\llvert V'(X^{i,n}_k)V'(X^{j,n}_k)\rrvert }{n^2}+
\frac
{\llvert V'(X^{i,n}_k)V'(X^{j,n}_k)\rrvert }{n^2\sqrt{\langle\mu
^n_{k/n},(V')^2\rangle}} \biggr).
\end{eqnarray*}
Plugging this estimate into (\ref{sumespcond}) and using the
boundedness of $g$ and (\ref{consexch}), one concludes that
\begin{eqnarray*}
&& \E \Biggl[ \Biggl(\frac{1}{n}\sum_{i=1}^n
\bigl(M^{i,n}_{\lceil nt\rceil
}-M^{i,n}_{\lceil ns\rceil} \bigr)g
\bigl(Y^{i,n}_{s_1},\ldots,Y^{i,n}_{s_p}
\bigr) \Biggr)^2 \Biggr]
\nonumber
\\
&&\qquad \leq C \Biggl(\frac{\lceil nt\rceil-\lceil ns\rceil}{n^{3/2}}+\frac
{1}{n^2}\sum
_{k=\lceil ns\rceil}^{\lceil nt\rceil-1} \Bigl(\E \bigl[ \bigl(V'
\bigl(Y^{i,n}_{k/n} \bigr) \bigr)^2 \bigr]+\sqrt
{\E \bigl[ \bigl(V' \bigl(Y^{i,n}_{k/n}
\bigr) \bigr)^2 \bigr]} \Bigr) \Biggr).
\end{eqnarray*}
One concludes with the local boundedness of $r\mapsto\sup_{n\geq
1}\sup_{1\leq i\leq n}\E[(V'(Y^{i,n}_r))^2]$ deduced from (\ref
{espvcar}) and exchangeability.
\end{pf*}

\begin{pf*}{Proof of Proposition~\ref{propffk}}
Since the function $\varphi$ is compactly supported and $V'$ is
continuous, one may suppose that $k$ is large enough so that $\forall
x\in\R$, $\llvert V'\varphi'(x)\rrvert \leq\llVert  \varphi'\rrVert  _\infty\sqrt
{(V'(x))^2\wedge k}$ and, therefore,
\[
\bigl\langle\mu^n_r,\bigl| V'\varphi'\bigr| \bigr\rangle\leq \bigl\llVert  \varphi'\bigr\rrVert  _\infty\sqrt{\bigl\langle\mu
^n_r,\bigl(V'\bigr)^2\wedge k\bigr\rangle}.
\]
By boundedness of $g$ and $\varphi''$,
then using (\ref{lipgam}) and (\ref{lipg}), one deduces
%
\begin{eqnarray}\label{difffk}
&& \E\bigl\llvert F_k \bigl(\mu^n \bigr)-F \bigl(
\mu^n \bigr)\bigr\rrvert\nonumber
\\
&&\qquad
\leq C\int_s^t
\E \Bigl[\bigl\llvert \Gamma \bigl( \bigl\langle \mu^n_r,
\bigl(V' \bigr)^2\wedge k \bigr\rangle, \bigl\langle
\mu^n_r,V'' \bigr\rangle
\bigr)-\Gamma \bigl( \bigl\langle \mu^n_r,
\bigl(V' \bigr)^2 \bigr\rangle, \bigl\langle
\mu^n_r,V'' \bigr\rangle
\bigr)\bigr\rrvert
\nonumber
\\
&&\hspace*{35pt}\quad\qquad{}+\bigl\llvert \cG \bigl( \bigl\langle\mu^n_r,
\bigl(V' \bigr)^2\wedge k \bigr\rangle, \bigl\langle\mu
^n_r,V'' \bigr\rangle
\bigr)-\cG \bigl( \bigl\langle\mu^n_r,
\bigl(V' \bigr)^2 \bigr\rangle, \bigl\langle\mu
^n_r,V'' \bigr\rangle
\bigr)\bigr\rrvert
\\
&&\hspace*{239pt}{}\times  \sqrt{ \bigl\langle\mu^n_r,
\bigl(V' \bigr)^2\wedge k \bigr\rangle} \Bigr]\,dr\hspace*{-8pt}\nonumber
\\
&&\qquad \leq C\int_s^t\E \Bigl[\sqrt{
\bigl\langle\mu^n_r, \bigl( \bigl(V'
\bigr)^2-k \bigr)^+ \bigr\rangle }+ \bigl\langle\mu^n_r,
\bigl( \bigl(V' \bigr)^2-k \bigr)^+ \bigr\rangle
\Bigr]\,dr.\nonumber
\end{eqnarray}
Since $\llvert V'(Y^{1,n}_r)\rrvert \leq\llvert V'(X^{1,n}_0)\rrvert +\llVert  V''\rrVert  _\infty
\llvert Y^{1,n}_r-Y^{1,n}_0\rrvert $, using the Cauchy--Schwarz and the Markov
inequalities, one obtains that
\begin{eqnarray*}
&&\E \bigl[ \bigl\langle\mu^n_r, \bigl(
\bigl(V' \bigr)^2-k \bigr)^+ \bigr\rangle \bigr]
\\
&&\qquad \leq\E \bigl[ \bigl(V' \bigl(Y^{1,n}_r \bigr)
\bigr)^21_{\{\llvert V'(Y^{1,n}_r)\rrvert \geq\sqrt{k}\}} \bigr]
\\
&&\qquad \leq 2\E \bigl[ \bigl( \bigl(V' \bigl(X^{1,n}_0
\bigr) \bigr)^2+\bigl\llVert V''\bigr
\rrVert ^2_\infty \bigl\llvert Y^{1,n}_r-Y^{1,n}_0
\bigr\rrvert ^2 \bigr)
\\
&&\hspace*{47pt}{}\times  (1_{\{\llvert V'(X^{1,n}_0)\rrvert \geq
({\sqrt{k}}/{2})\}}+1_{\{\llvert Y^{1,n}_r-Y^{1,n}_0\rrvert \geq ({\sqrt{k}}/({2\llVert
V''\rrVert  _\infty}))\}} ) \bigr]
\\
&&\qquad \leq\frac{C}{k} \bigl(\E \bigl[ \bigl(V'
\bigl(X^{1,n}_0 \bigr) \bigr)^4 \bigr]
\\
&&\hspace*{45pt}{} +\E
^{1/2} \bigl[\bigl\llvert Y^{1,n}_r-Y^{1,n}_0
\bigr\rrvert ^4 \bigr]\E^{1/2} \bigl[ \bigl(V'
\bigl(X^{1,n}_0 \bigr) \bigr)^4 \bigr]
+\E
\bigl[\bigl\llvert Y^{1,n}_r-Y^{1,n}_0
\bigr\rrvert ^4 \bigr] \bigr).
\end{eqnarray*}
Therefore, by (\ref{majmomacc}),%
\begin{equation}
\lim_{k\to\infty}\sup_{n\geq1}\sup
_{r\in[0,t]}\E \bigl[ \bigl\langle\mu ^n_r,
\bigl( \bigl(V' \bigr)^2-k \bigr)^+ \bigr\rangle
\bigr]=0.\label{cutoff}
\end{equation}
One concludes by plugging this result into~(\ref{difffk}).
\end{pf*}

\section{Proof of Proposition~\texorpdfstring{\protect\ref{proplimiting-acceptance}}{1}}\label{secprop}
By (\ref{reste}) and~\cite{roberts-gelman-gilks-97}, Proposition 2.4,
which is also a consequence of~(\ref{calesp3}) for the choice $\alpha
=0$, there is a finite deterministic constant~$C$ not depending on $t$
such that
\[
\biggl\llvert \mathbb{P} \bigl({\mathcal A}_{\lfloor nt\rfloor+1}|{\mathcal
F}^n_{\lfloor nt\rfloor} \bigr)-\frac{1}{l^2}\Gamma \bigl( \bigl
\langle \mu^n_{{\lfloor nt\rfloor}/{n}}, \bigl(V'
\bigr)^2 \bigr\rangle, \bigl\langle\mu^n_{{\lfloor
nt\rfloor}/{n}},V''
\bigr\rangle \bigr) \biggr\rrvert \leq\frac{C}{\sqrt{n}}.
\]
With (\ref{lipgam}), one deduces that
%
\begin{eqnarray}\label{contprobacc}
&& \E \biggl\llvert \mathbb{P} \bigl({\mathcal A}_{\lfloor nt\rfloor+1}|{\mathcal
F}^n_{\lfloor nt\rfloor} \bigr)-\frac{1}{l^2}\Gamma \bigl(\E \bigl[
\bigl(V'(X_t) \bigr)^2 \bigr],\E
\bigl[V''(X_t) \bigr] \bigr) \biggr\rrvert
\nonumber
\\
&&\qquad \leq C \biggl(\frac{1}{\sqrt{n}}+ \bigl(\E +\E^{1/2} \bigr) \bigl
\llvert \bigl\langle\mu^n_{{\lfloor nt\rfloor
}/{n}}, \bigl(V'
\bigr)^2 \bigr\rangle-\E \bigl[ \bigl(V'(X_t)
\bigr)^2 \bigr] \bigr\rrvert
\\
&&\hspace*{127pt}{} +\E \bigl\llvert \bigl\langle\mu
^n_{{\lfloor nt\rfloor}/{n}},V'' \bigr\rangle-\E
\bigl[V''(X_t) \bigr] \bigr\rrvert\biggr).\nonumber
\end{eqnarray}
One has for $k\in{\mathbb N}$,
\begin{eqnarray*}
&& \E \bigl\llvert \bigl\langle\mu^n_{{\lfloor nt\rfloor}/{n}},
\bigl(V' \bigr)^2 \bigr\rangle -\E \bigl[
\bigl(V'(X_t) \bigr)^2 \bigr] \bigr\rrvert
\\
&&\qquad \leq \E \bigl\llvert \bigl\langle\mu^n_{{\lfloor
nt\rfloor}/{n}}-
\mu^n_{t}, \bigl(V' \bigr)^2
\bigr\rangle \bigr\rrvert +\E \bigl\langle \mu ^n_{t},
\bigl( \bigl(V' \bigr)^2-k \bigr)^+ \bigr\rangle
\\
&&\quad\quad{} +\E \bigl\llvert \bigl\langle\mu ^n_{t},
\bigl(V' \bigr)^2\wedge k \bigr\rangle-\E \bigl[
\bigl(V'(X_t) \bigr)^2\wedge k \bigr] \bigr
\rrvert +\E \bigl[ \bigl( \bigl(V' \bigr)^2-k
\bigr)^+(X_t) \bigr].
\end{eqnarray*}
By the end of the proof of Proposition~\ref{intmart} [see in
particular (\ref{majoecdisc})], the first term in the right-hand side
converges to $0$ locally uniformly in $t$ as $n\to\infty$. By (\ref
{cutoff}) and Theorem~\ref{chaos}, the sum of the second and last
terms in the right-hand side converges to $0$ as $k\to\infty$
uniformly in $n$ and locally uniformly in $t$. Last, for fixed $k$, the
third term converges to $0$ as $n\to\infty$ locally uniformly in $t$
by Theorem~\ref{chaos}. One deduces that $\E\llvert \langle\mu
^n_{{\lfloor nt\rfloor}/{n}},(V')^2\rangle-\E
[(V'(X_t))^2]\rrvert $ converges\vspace*{1pt} to $0$ as $n\to\infty$ locally
uniformly in $t$.
Dealing with the other expectation in the right-hand side of (\ref
{contprobacc}) in a similar but easier way (since $V''$ is bounded),
one completes the proof.

\begin{appendix}\label{secapp}
\section*{Appendix: Proofs of technical results}
\setcounter{equation}{0}

In this section, we first give a proof of Lemma~\ref{propggam} which
gives basic properties of the functions $\Gamma$ and $\cG$. Then we
give some explicit formulas for some expectations involving Gaussian
random variables.

\begin{pf*}{Proof of Lemma~\ref{propggam}}
The functions $\cG$ and $\Gamma$ are clearly continuous on
$(0,+\infty)\times\R$. We recall the usual tail estimate for the
Normal law: $\forall x>0$,
%
\begin{equation}
\Phi(-x)=\int_x^{+\infty}e^{-{y^2}/{2}}
\frac{dy}{\sqrt{2\pi
}} \leq\int_x^{+\infty}
\frac{y}{x}e^{-y^2/2}\frac
{dy}{\sqrt{2\pi}}=\frac{e^{-x^2/2}}{x\sqrt{2\pi}}.\label{tailnorm}
\end{equation}
One deduces that for $a>b^+$,
%
\begin{eqnarray}\label{majog2}
\Phi \biggl(l \biggl(\frac{b}{2\sqrt{a}}-\sqrt{a} \biggr) \biggr)
&\leq& \frac{2}{l\sqrt{2\pi a}}e^{-(l^2(b-2a)^2)/(8a)}\quad\mbox{and}
\nonumber\\[-8pt]\\[-8pt]\nonumber
{\mathcal G}(a,b)&\leq& \frac{2l}{\sqrt{2\pi a}}e^{-(l^2b^2)/(8a)}.
\end{eqnarray}
Since for $0\leq a\leq b$, $\cG(a,b)\leq l^2\times1\times1$, one
deduces (\ref{majog}). Moreover, (\ref{majog2}) implies that $\cG$
is continuous on $\{(0,+\infty]\times\R\}\cup\{\{0\}\times(-\infty,0)\}$. With the continuity of $(a,b)\mapsto\frac{b}{\sqrt{a}}$ on
$(0,+\infty]\times\R$ under the convention $\frac{b}{\sqrt{\infty
}}=0$, one deduces that $\Gamma$ is continuous on $(0,+\infty]\times
\R$.
For $\beta>0$, $\lim_{a\to0^+,b\to\beta}\Phi (\frac
{b}{2\sqrt{a}}-\sqrt{a} )=1$ and,\vspace*{1pt} therefore, $\lim_{a\to
0^+,b\to\beta}{\cG}(a,b)={\cG}(0,\beta)$, which completes the
proof of the continuity properties of $\cG$.
Since for $(a,b)\in(0,+\infty)\times\R$, $\partial_b \Gamma
(a,b)=-\frac{l^4}{2}e^{({l^2(a-b)})/{2}}\Phi (l (\frac
{b}{2\sqrt{a}}-\sqrt{a} ) )<0$, for\vspace*{1pt} fixed $a\in
(0,+\infty)$, the function $b\mapsto\Gamma(a,b)$ is decreasing. One
easily checks that\vspace*{1pt} for fixed $b<0$, $\lim_{a\to0^+}\Gamma
(a,b)=l^2+0=\Gamma(0,b)$ and\vspace*{1pt} for fixed $b>0$,\break $\lim_{a\to0^+}\Gamma
(a,b)=0+l^2e^{-(l^2b)/2}=\Gamma(0,b)$. With the previous
monotonicity property, one deduces that $\lim_{a\to0^+}\Gamma
(a,0)=l^2=\Gamma(0,0)$. The continuity of $b\mapsto\Gamma(0,b)$ and
Dini's lemma implies that $b\mapsto\Gamma(a,b)$ converges locally
uniformly to $b\mapsto\Gamma(0,b)$ as $a\to0^+$ and that~$\Gamma$
is continuous on $[0,+\infty]\times\R$.
Since $\Gamma$ is positive on $[0,+\infty]\times\R$, one deduces
that (\ref{mingam}) holds.
For $a>0$, by~(\ref{majog2}), $\lim_{b\to-\infty}{\cG}(a,b)=0$.
Since $\lim_{b\to-\infty}\Phi (-\frac{l b}{2\sqrt{a}}
)=1$, one deduces that $\lim_{b\to-\infty}\Gamma(a,b)=l^2$. By
monotonicity of $b\mapsto\Gamma(a,b)$, one deduces that $\forall
(a,b)\in(0,+\infty)\times\R$, $\Gamma(a,b)\leq l^2$. This bound
still holds for $a\in\{0,+\infty\}$ by continuity (or using the
explicit expression of $\Gamma$).
For $(a,b)\in(0,+\infty)\times\R$, one has
\begin{eqnarray*}
\partial_b \Gamma(a,b)&=&-\frac{l^2}{2}\cG(a,b),
\\
\partial_a \Gamma(a,b)&=&\frac{l^2}{2}\cG(a,b)-
\frac{l^3}{2\sqrt
{2\pi a}}e^{-(l^2b^2)/(8a)},
\\
\partial_b \cG(a,b)&=&-\frac{l^2}{2}\cG(a,b)+
\frac{l^3}{2\sqrt{2\pi a}}e^{-(l^2b^2)/(8a)},
\\
\partial_a \cG(a,b)&=&\frac{l^2}{2}\cG(a,b)-
\frac{l^3}{2\sqrt{2\pi
}} \biggl(\frac{1}{\sqrt{a}}+\frac{b}{2a^{3/2}}
\biggr)e^{-(l^2b^2)/(8a)}.
\end{eqnarray*}
The boundedness of $\cG$ then implies (\ref{lipgam}). Concerning
(\ref{lipg}), let us give some details for the inequality
\[
\bigl(\sqrt{a} \wedge\sqrt{a'} \bigr)\bigl\llvert \cG(a,b) - \cG
\bigl(a',b \bigr)\bigr\rrvert \le C \bigl(\bigl\llvert
a'-a\bigr\rrvert +\bigl\llvert \sqrt{a'}-\sqrt{a}
\bigr\rrvert \bigr).
\]
Let us assume that $0 \le a < a'$ and $b\in[\inf V'',\sup V'']$. Then
we have
\begin{eqnarray*}
&& \bigl(\sqrt{a} \wedge\sqrt{a'} \bigr)\bigl\llvert \cG(a,b) - \cG
\bigl(a',b \bigr)\bigr\rrvert
\\
&&\qquad =\sqrt{a} \biggl\llvert \int
_a^{a'} \partial_a \cG(x,b) \,dx
\biggr\rrvert
\\
&&\qquad =\sqrt{a} \biggl\llvert \int_a^{a'}
\frac{l^2}{2}\cG(x,b)-\frac{l^3}{2\sqrt{2\pi}} \biggl(\frac
{1}{\sqrt{x}}+
\frac{b}{2x^{3/2}} \biggr)e^{-(l^2b^2)/(8x)} \,dx \biggr\rrvert
\\
&&\qquad \le C \sqrt{a}\int_a^{a'} \biggl(
\frac{1}{\sqrt{x}}+\frac
{1}{x} \biggr) \,dx \le C \biggl(
\bigl(a'-a \bigr) + \int_a^{a'}
\frac{\sqrt{a}}{x} \,dx \biggr)
\\
&&\qquad \le C \biggl( \bigl(a'-a \bigr) + \int_a^{a'}
\frac{1}{\sqrt{x}} \,dx \biggr) \le C \bigl( \bigl(a'-a \bigr) +
\bigl(\sqrt{a'}-\sqrt{a} \bigr) \bigr),
\end{eqnarray*}
where we used (\ref{majog}) and the boundedness of $(x,b) \in
(0,+\infty] \times\R
\mapsto\break \frac{b}{2\sqrt{x}}e^{-(l^2b^2)/(8x)}$ for the first inequality.
\end{pf*}

%
\begin{lemma}\label{calesp}For $\alpha,\beta,\gamma,\delta\in\R$
and independent normal random variables $G$, $\tilde{G}$ and $\hat
{G}$, one has
%
\begin{eqnarray}
\qquad && \E \bigl(G \bigl(e^{\alpha G+\beta\tilde{G}+\gamma}\wedge 1 \bigr) \bigr)
\nonumber\\[-8pt]\label{calesp1}  \\[-8pt]\nonumber
&&\qquad =\alpha e^{\gamma+((\alpha^2+\beta^2)/2)}
\Phi \biggl(-\frac{\gamma
+\alpha^2+\beta^2}{\sqrt{\alpha^2+\beta^2}} \biggr)
=\frac{\alpha}{l^2}\cG \biggl(\frac{\alpha^2+\beta^2}{l^2},-\frac
{2\gamma}{l^2}
\biggr),\nonumber
\\
&& \bigl\llvert \E \bigl(G \bigl(1-e^{\alpha
G+\beta\tilde{G}+\gamma} \bigr)^+ \bigr) \bigr\rrvert
\leq \biggl(\sqrt{\frac{2}{\pi}}+\sqrt{2\gamma^-} \biggr)\sqrt{
\frac
{\alpha^2}{\alpha^2+\beta^2}},\label{calesp2}
\\
&& \E \bigl(G^2 \bigl(e^{\alpha G+\beta\tilde{G}+\gamma}\wedge 1 \bigr) \bigr)\nonumber
\\
\label{calesp3} &&\qquad = \bigl(1+ \alpha^2 \bigr)e^{\gamma+((\alpha^2+\beta^2)/2)}\Phi \biggl(-\frac{\gamma+\alpha^2+\beta^2}{\sqrt{\alpha
^2+\beta^2}}
\biggr)
\\
&&\quad\qquad{} +\Phi \biggl(\frac{\gamma}{\sqrt{\alpha^2+\beta^2}} \biggr)-\frac
{\alpha^2}{\sqrt{2\pi(\alpha^2+\beta^2)}}e^{-(\gamma^2/(2(\alpha^2+\beta^2))},\nonumber
\\
&&\E \bigl(G\hat{G} \bigl(e^{\alpha
G+\beta\tilde{G}+\delta\hat{G}+\gamma}\wedge 1 \bigr) \bigr)
\nonumber
\\
\label{calesp4} &&\qquad =\alpha\delta \biggl(e^{\gamma+((\alpha^2+\beta^2+\delta^2)/2)}\Phi \biggl(-\frac{\gamma+\alpha^2+\beta^2+\delta^2}{\sqrt{\alpha^2+\beta
^2+\delta^2}} \biggr)
\\
&&\hspace*{137pt}{} -
\frac{e^{-(\gamma^2/({2(\alpha^2+\beta
^2+\delta^2)}))}}{\sqrt{2\pi(\alpha^2+\beta^2+\delta^2)}} \biggr),\nonumber
\\
&& \forall a \in[0,+\infty),\qquad \E \bigl(\cG(a,\alpha G+\beta) \bigr)=\cG \biggl(a+
\frac{l^2\alpha^2}{4},\beta \biggr).\label{calesp5}
\end{eqnarray}
\end{lemma}

\begin{pf}In this proof, the identity $\E(f(G)e^{\alpha
G - \alpha^2/2})=\E(f(\alpha+G))$ is repeatedly used. Let us start
with~(\ref{calesp1}).
By the symmetry of the normal law, $\alpha\mapsto\E (G
(e^{\alpha G+\beta\tilde{G}+\gamma}\wedge1 ) )$ is an
odd function and we only need to check~(\ref{calesp1}) for $\alpha>0$.
Conditioning by $\tilde{G}$ for the third equality, we get
\begin{eqnarray*}
&& \E \bigl(G \bigl(e^{\alpha G+\beta\tilde{G}+\gamma}\wedge1 \bigr) \bigr)
\\
&&\qquad =\E \bigl(e^{\gamma+(\alpha^2/2)}e^{\alpha G-(\alpha^2/2)}
e^{\beta\tilde{G}}G1_{\{G\leq-(\gamma+\beta
\tilde{G})/\alpha\}}+G1_{\{G>(\gamma+\beta\tilde{G})/{\alpha
}\}}
\bigr)
\\
&&\qquad = e^{\gamma+(\alpha^2/2)}\E \bigl(e^{\beta
\tilde{G}}(\alpha+G)1_{\{\alpha+G\leq-(\gamma+\beta\tilde
{G})/\alpha\}} \bigr)+
\E (G1_{\{G>(\gamma+\beta\tilde
{G})/{\alpha}\}} )
\\
&&\qquad =\alpha e^{\gamma+((\alpha^2+\beta^2)/2)}\mathbb{P} \biggl(\frac{\alpha
G+\beta(\beta+\tilde{G})}{\sqrt{\alpha^2+\beta^2}}\leq -
\frac{\gamma+\alpha^2}{\sqrt{\alpha^2+\beta^2}} \biggr)
\\
&&\quad\qquad{} -\frac {e^{\gamma+(\alpha^2/2)}}{\sqrt{2\pi}}\E \bigl(e^{\beta
\tilde{G}}e^{-{(\gamma+\alpha^2+\beta\tilde{G})^2}/(2\alpha^2)}
\bigr)
\\
&&\quad\qquad{}
+\frac{1}{\sqrt{2\pi}}\E \bigl(e^{-{(\gamma+\beta\tilde
{G})^2}/({2\alpha^2})} \bigr).
\end{eqnarray*}
We deduce~(\ref{calesp1}) by remarking that the two last terms
compensate each other since
\[
\gamma+\frac{\alpha^2}{2}+\beta\tilde{G}-\frac{(\gamma+\alpha
^2+\beta\tilde{G})^2}{2\alpha^2}=-
\frac{(\gamma+\beta\tilde
{G})^2}{2\alpha^2}.
\]
To obtain the inequality~(\ref{calesp2}), we notice that
\begin{eqnarray*}
&& \E \bigl(G \bigl(1-e^{\alpha G+\beta\tilde{G}+\gamma} \bigr)^+ \bigr)
\\
&&\qquad =\E \bigl(G
\bigl(1-e^{\alpha G+\beta\tilde{G}+\gamma
} \bigr)^+ \bigr)-\E(G)
\\
&&\qquad =-\E \bigl(G \bigl(e^{\alpha G+\beta\tilde{G}+\gamma}\wedge 1 \bigr) \bigr)
\\
&&\qquad =-\frac{\alpha}{l\sqrt{\alpha^2+\beta^2}}
\times \sqrt{\frac{\alpha^2+\beta^2}{l^2}}\cG \biggl(\frac{\alpha
^2+\beta^2}{l^2},-
\frac{2\gamma}{l^2} \biggr)
\end{eqnarray*}
and conclude using (\ref{majog}).
To derive~(\ref{calesp3}), one obtains by conditioning by $G$ for the
second equality
%
\begin{eqnarray}\label{eg2ind}
&& \E \bigl(G^2 \bigl(e^{\alpha G+\beta\tilde{G}+\gamma}\wedge 1 \bigr) \bigr)
\nonumber
\\
&&\qquad =e^{((\alpha^2+\beta^2)/2)+\gamma}\E \bigl(G^2e^{\alpha
G+\beta\tilde{G}-((\alpha^2+\beta^2)/2)}1_{\{\alpha G+\beta
\tilde{G}\leq-\gamma\}}
\bigr)\nonumber
\\
&&\quad\qquad{} +\E \biggl(G^2\Phi \biggl(\frac
{\gamma+\alpha G}{\llvert \beta\rrvert } \biggr) \biggr)
\\
&&\qquad =e^{((\alpha^2+\beta^2)/2)+\gamma}\E \biggl( \bigl(G^2+2\alpha G+
\alpha^2 \bigr)\Phi \biggl(-\frac{\gamma+\alpha G+\alpha^2+\beta
^2}{\llvert \beta\rrvert } \biggr) \biggr)\nonumber
\\
&&\quad\qquad{} +\E
\biggl(G^2\Phi \biggl(\frac{\gamma
+\alpha G}{\llvert \beta\rrvert } \biggr)
\biggr).\nonumber
\end{eqnarray}
By integration by parts,
%
\begin{eqnarray}\label{eg2n}
\qquad && \E \biggl(G^2\Phi \biggl(\frac{\gamma+\alpha
G}{\llvert \beta\rrvert } \biggr) \biggr)\nonumber
\\
&&\qquad =
\frac{1}{\sqrt{2\pi}}\int_\R x^2\Phi \biggl(
\frac{\gamma+\alpha
x}{\llvert \beta\rrvert } \biggr)e^{-x^2/2}\,dx\nonumber
\\
&&\qquad =\frac{1}{\sqrt{2\pi}}\int_\R\Phi \biggl(
\frac{\gamma+\alpha x}{\llvert \beta\rrvert } \biggr)e^{-x^2/2}\,dx\nonumber
\\
&&\quad\qquad{}+\frac{\alpha}{2\pi\llvert \beta
\rrvert }\int
_\R xe^{-(x^2/2)-((\gamma+\alpha x)^2/(2\beta^2))}\,dx
\\
&&\qquad ={\mathbb P} \bigl(\llvert \beta\rrvert \tilde{G}-\alpha G\leq\gamma \bigr)\nonumber
\\
&&\quad\qquad{}+
\frac{\alpha e^{-{\gamma^2}/(2(\alpha^2+\beta^2))}}{2\pi
\llvert \beta\rrvert }\int_\R xe^{-({(\alpha^2+\beta^2)(x+((\gamma\alpha)/(\alpha^2+\beta^2)))^2})/(2\beta^2)}\,dx
\nonumber
\\
&&\qquad =\Phi \biggl(\frac{\gamma}{\sqrt{\alpha^2+\beta^2}} \biggr)-e^{-\gamma^2/({2(\alpha^2+\beta^2)})}\frac{\alpha^2\gamma
}{\sqrt{2\pi(\alpha^2+\beta^2)^3}}\nonumber
\end{eqnarray}
and
%
\begin{eqnarray}\label{gng}
&& \E \biggl(G\Phi \biggl(-\frac{\gamma+\alpha G+\alpha^2+\beta
^2}{\llvert \beta\rrvert } \biggr) \biggr)\nonumber
\\
&&\qquad
=-
\frac{\alpha}{2\pi\llvert \beta\rrvert }\int_\R e^{-(x^2/2)-((\alpha x+\gamma+\alpha^2+\beta
^2)^2/(2\beta^2))}\,dx
\\
&&\qquad =-\frac{\alpha}{\sqrt{2\pi(\alpha^2+\beta^2)}}e^{-{(\gamma
+\alpha^2+\beta^2)^2}/(2(\alpha^2+\beta^2))}.\nonumber
\end{eqnarray}
One obtains~(\ref{calesp3}) by plugging this last equality together
with (\ref{eg2n}) also written with $(\alpha,\gamma)$ replaced by
$(-\alpha,-(\gamma+\alpha^2+\beta^2))$ in (\ref{eg2ind}).

To prove~(\ref{calesp4}), conditioning by $\hat{G}$, using (\ref
{calesp1}) and then (\ref{gng}), one obtains
\begin{eqnarray*}
&& \E \bigl(G\hat{G} \bigl(e^{\alpha G+\beta\tilde{G}+\delta\hat
{G}+\gamma}\wedge1 \bigr) \bigr)
\\
&&\qquad =\alpha
e^{\gamma+(({\alpha^2+\beta^2})/{2})}\E \biggl(\hat{G}e^{\delta\hat{G}}\Phi \biggl(-
\frac{\gamma+\delta\hat{G}+\alpha^2+\beta^2}{\sqrt{\alpha^2+\beta
^2}} \biggr) \biggr)
\\
&&\qquad =\alpha e^{\gamma+(({\alpha^2+\beta^2+\delta^2})/{2})}\E \biggl((\hat{G}+\delta)\Phi \biggl(-
\frac{\gamma+\delta\hat{G}+\alpha
^2+\beta^2+\delta^2}{\sqrt{\alpha^2+\beta^2}} \biggr) \biggr)
\\
&&\qquad =\alpha\delta e^{\gamma+(({\alpha^2+\beta^2+\delta
^2})/{2})}
\\
&&\quad\qquad{}\times  \biggl(\Phi \biggl(-\frac{\gamma+\alpha^2+\beta^2+\delta
^2}{\sqrt{\alpha^2+\beta^2+\delta^2}}
\biggr)-\frac{e^{-{(\gamma+\alpha^2+\beta^2+\delta^2)^2}/(2(\alpha^2+\beta^2+\delta
^2))}}{\sqrt{2\pi(\alpha^2+\beta^2+\delta^2)}} \biggr).
\end{eqnarray*}

Last,
\begin{eqnarray*}
&& \frac{1}{l^2}\E \bigl(\cG(a,\alpha G+\beta) \bigr)
\\
&&\qquad = e^{(l^2(a+l^2\alpha^2/4-\beta))/2}
\mathbb{P} \biggl(\tilde{G}\leq l \biggl(\frac{\alpha G-l^2\alpha^2/2+\beta}{2\sqrt{a}}-\sqrt {a} \biggr)
\biggr)
\\
&&\qquad =e^{(l^2(a+l^2\alpha^2/4-\beta))/2}\mathbb{P} \biggl(\frac
{\sqrt{a+l^2\alpha^2/4}}{\sqrt{a}} \hat G\leq l
\frac{\beta-2(a+l^2\alpha^2/4)}{2\sqrt{a}} \biggr),
\end{eqnarray*}
which yields~(\ref{calesp5}).
\end{pf}

To prove Lemma~\ref{propdrift}, we need the following lemma.

\begin{lemma}\label{wass}
Let $X$, $Y$ denote two real random variables with respective
cumulative distribution functions $F_X$ and $F_Y$ and $f\dvtx \R\to\R$ be
a bounded function, Lipschitz continuous with constant $L(f)$ outside
$[-\varepsilon,\varepsilon]$ for some constant $\varepsilon>0$. If
$X$ admits a bounded density $p_X$ with respect to the Lebesgue measure
on $\R$, then
\begin{eqnarray*}
&& \bigl\llvert \E \bigl[f(X) \bigr]-\E \bigl[f(Y) \bigr]\bigr\rrvert
\\
&&\qquad \leq
L(f)W_1(X,Y)
+2(\sup f-\inf f) \bigl(\sqrt{2\llVert p_X\rrVert
_\infty W_1(X,Y)}+\llVert p_X\rrVert
_\infty\varepsilon \bigr),
\end{eqnarray*}
where $W_1(X,Y)=\inf_{(Z,W)\dvtx Z\stackrel{(d)}=X,W\stackrel{(d)}=Y}\E
\llvert Z-W\rrvert $ denotes the Wasserstein distance between the laws of $X$ and $Y$.
\end{lemma}

\begin{pf} 
Let for $u\in(0,1)$, $F_X^{-1}(u)=\inf\{x\in\R\dvtx F_X(x)\geq u\}$ denote
the c\`{a}g pseudo-inverse of $F_X$ and $F_Y^{-1}$ be defined in the same
way. Then $\forall x\in\R$, $\forall u\in(0,1)$, $F_X^{-1}(u)\leq
x\Leftrightarrow u\leq F_X(x)$. Moreover, if $U$ is uniformly
distributed on $[0,1]$, then $F_X^{-1}(U)\stackrel{(d)}{=}X$,
$F_Y^{-1}(U)\stackrel{(d)}{=}Y$ and according to \cite{rachev-ruschendorf-98}, \mbox{pages~107--109}, $W_1(X,Y)=\E
\llvert F_X^{-1}(U)-F_Y^{-1}(U)\rrvert $. As a consequence,
\begin{eqnarray*}
&& \bigl\llvert \E \bigl[f(X) \bigr]-\E \bigl[f(Y) \bigr]\bigr\rrvert
\\
&&\qquad = \bigl
\llvert \E \bigl[f \bigl(F_X^{-1}(U) \bigr)-f
\bigl(F_Y^{-1}(U) \bigr) \bigr]\bigr\rrvert
\\
&&\qquad \leq \bigl\llvert \E \bigl[ \bigl(f \bigl(F_X^{-1}(U)
\bigr)-f \bigl(F_Y^{-1}(U) \bigr) \bigr)
\\
&&\hspace*{46pt}{}\times
(1_{\{F_X^{-1}(U)\vee
F_Y^{-1}(U)\leq-\varepsilon\}}+1_{\{F_X^{-1}(U)\wedge
F_Y^{-1}(U)>\varepsilon\}}) \bigr]\bigr\rrvert
\\
&&\quad\qquad{}+\bigl\llvert \E \bigl[ \bigl(f \bigl(F_X^{-1}(U)
\bigr)-f \bigl(F_Y^{-1}(U) \bigr) \bigr)
\\
&&\hspace*{58pt}{}\times
(1_{\{F_X^{-1}(U)\leq
-\varepsilon<F_Y^{-1}(U)\}}+1_{\{F_X^{-1}(U)>\varepsilon\geq
F_Y^{-1}(U)\}}) \bigr]\bigr\rrvert
\\
&&\quad\qquad{} +\bigl\llvert \E \bigl[ \bigl(f \bigl(F_X^{-1}(U)
\bigr)-f \bigl(F_Y^{-1}(U) \bigr) \bigr)1_{\{-\varepsilon
<F_X^{-1}(U)\leq\varepsilon\}}
\bigr]\bigr\rrvert
\\
&&\qquad \leq  L(f)\E\bigl\llvert F_X^{-1}(U)-F_Y^{-1}(U)
\bigr\rrvert
\\
&&\quad\qquad{}+(\sup f-\inf f)
\bigl(\mathbb{P} \bigl(F_Y(-\varepsilon)<U
\leq F_X(-\varepsilon) \bigr)
\\
&&\hspace*{115pt}{}
+\mathbb{P} \bigl(F_X(
\varepsilon)<U\leq F_Y(\varepsilon) \bigr)
\\
&&\hspace*{115pt}{}+\mathbb{P} \bigl(F_X(-\varepsilon)<U\leq F_X(
\varepsilon) \bigr) \bigr)
\\
&&\qquad =L(f)W_1(X,Y)
\\
&&\quad\qquad{}+(\sup f-\inf f) \biggl( \bigl(F_X(-\varepsilon)-F_Y(-
\varepsilon ) \bigr)^++ \bigl(F_Y(\varepsilon)-F_X(
\varepsilon) \bigr)^+
\\
&&\hspace*{241pt}{}+\int_{-\varepsilon
}^\varepsilon
p_X(x)\,dx \biggr).
\end{eqnarray*}
One concludes by using the inequality
\[
\sup_{x\in\R}\bigl\llvert F_X(x)-F_Y(x)
\bigr\rrvert \leq \sqrt{2\llVert p_X\rrVert _\infty
W_1(X,Y)}.
\]
This inequality is stated in \cite{mattingly-pillai-stuart-12}, Lemma~5.4, with the factor $2$ replaced by $4$
but a careful look at the proof of this lemma shows that it holds with
the factor $2$.
\end{pf}

\begin{pf*}{Proof of Lemma~\ref{propdrift}}
By Lipschitz continuity of $x\mapsto e^{x}\wedge1$ and the Taylor expansion
\[
V \biggl(x_i+\frac{l}{\sqrt{n}}G^i \biggr)=V(x_i)+
\frac{l
V'(x_i)}{\sqrt{n}}G^i+\frac{l^2V''(x_i)}{2n} \bigl(G^i
\bigr)^2+\frac
{l^3V^{(3)}(\chi_i)}{6n^{3/2}} \bigl(G^i
\bigr)^3
\]
with $\chi_i\in[x_i,x_i+\frac{l}{\sqrt{n}}G^i]$, one obtains
\begin{eqnarray*}
&&\E \bigl[ \bigl(e^{\sum_{i=1}^n (V(x_i)-V(x_i+(l/{\sqrt
{n}})G^i))}\wedge1-e^{-\sum_{i=1}^n ((l/\sqrt
{n})V'(x_i)G^i+(l^2/(2n))V''(x_i))}\wedge1
\bigr)^2 \bigr]
\\
&&\qquad \leq\E \Biggl[ \Biggl(\sum_{i=1}^n
\biggl(\frac
{l^2V''(x_i)}{2n} \bigl( \bigl(G^i \bigr)^2-1
\bigr)+ \frac{l^3{V^{(3)}}(\chi
_i)}{6n^{3/2}} \bigl(G^i \bigr)^3 \biggr)
\Biggr)^2 \Biggr].
\end{eqnarray*}
Developing\vspace*{1pt} the square and remarking that for $i\neq j$, $\E
[((G^i)^2-1)((G^j)^2-1) ]=0=\E [((G^i)^2-1)V^{(3)}(\chi
_j)(G^j)^3 ]$, one easily deduces (\ref{reste}) using the
boundedness of $V''$ and $V^{(3)}$.

The\vspace*{1pt} proof of the two other inequalities is inspired by~\cite{mattingly-pillai-stuart-12}, Section~5, where the authors first replace
$V(x_1)-V(x_1+\frac{l}{\sqrt{n}}G^1)$ by $-\frac{lV'(x_1)}{\sqrt{n}}G^1$
in the exponential factor at a cost ${\mathcal O}(\frac{1}{n})$. Then
they explicitly compute the conditional expectation given $(G^2,\ldots,G^n)$ to improve the regularity of the function in the expectation.
Next, they replace $\sum_{i=2}^n (V(x_i+\frac{l}{\sqrt
{n}}G^i)-V(x_i))$ by the Gaussian random variable $\sum_{i=2}^n (\frac
{l V'(x_i)}{\sqrt{n}}G^i+\frac{l^2 V''(x_i)}{2n})$ and control the
resulting error by some Wasserstein distance estimate between these two
random variables. To preserve symmetry in the estimate and in
particular to obtain $\langle\nu_n,(V')^2\rangle$ instead of $\frac
{1}{n}\sum_{i=2}^n(V'(x_i))^2$ in the denominators, we write $G_1$ as
the sum of two independent variables distributed according to $\mathcal
{N}(0,\frac{1}{2})$.

Let $\tilde{G}^1=\frac{G^1}{\sqrt{2}}$, $\tilde{G}^i=G^i$ for
$i\geq2$ and $\hat{G}^1\sim\mathcal{N}(0,\frac{1}{2})$ be
independent from $(G^1,\ldots,G^n)$. One has
\begin{eqnarray*}
&&\E \bigl(G^1 \bigl(e^{\sum_{i=1}^n
(V(x_i)-V(x_i+ (l/\sqrt{n})G^i))}\wedge1 \bigr) \bigr)
\\
&&\qquad =2 \E \bigl(\hat{G}^1 \bigl(e^{V(x_1)-V(x_1+({l}/{\sqrt{n}})(\tilde
{G}^1+\hat{G}^1))+\sum_{i=2}^n
(V(x_i)-V(x_i+(l/\sqrt{n})\tilde{G}^i))}\wedge 1 \bigr)
\bigr).
\end{eqnarray*}
As in the above derivation of (\ref{reste}), one deduces from the
Lipschitz continuity of $y\mapsto e^y\wedge1$ and the boundedness of
$V''$ that
\[
\bigl\llvert \E \bigl(G^1 \bigl(e^{\sum_{i=1}^n (V(x_i)-V(x_i+(l/\sqrt{n})G^i))}\wedge1 \bigr)
\bigr)-E \bigr\rrvert \leq\frac{C}{n},
\]
where, by conditioning by $(\tilde{G}^1,\ldots,\tilde{G}^n)$ and
using (\ref{calesp1}),
\begin{eqnarray*}
E&\stackrel{\mathrm{def}}{=}&2\E \bigl(\hat{G}^1 \bigl(e^{-((l
V'(x_1))/{\sqrt{n}})\hat{G}^1+\sum_{i=1}^n  (V(x_i)-V
(x_i+(l/{\sqrt{n}})\tilde{G}^i ) )}
\wedge1 \bigr) \bigr)
\\
&=&-\frac{V'(x_1)}{l\sqrt{n}}\E \Biggl[{\cG} \Biggl(\frac
{(V'(x_1))^2}{2n},
\frac{2}{l^2}\sum_{i=1}^n \biggl(V
\biggl(x_i+\frac
{l}{\sqrt{n}}\tilde{G}^i
\biggr)-V(x_i) \biggr) \Biggr) \Biggr].
\end{eqnarray*}
By boundedness of $\cG$ and since
\begin{eqnarray*}
&&\E \bigl[G^1 \bigl\{ \bigl(e^{\sum_{i=1}^n (V(x_i)-V(x_i+({l}/{\sqrt{n}})G^i))}\wedge1 \bigr)+
\bigl(1-e^{\sum_{i=1}^n
(V(x_i)-V(x_i+({l}/{\sqrt{n}})G^i))} \bigr)^+ \bigr\} \bigr]
\\
&&\qquad =\E[G_1]=0,
\end{eqnarray*}
one deduces (\ref{restrev1}).

Moreover, when $V'(x_1)=0$, $E=0$ and (\ref{restrev2}) holds. To deal
with the case $V'(x_1)\neq0$, we let
\begin{eqnarray*}
X&\stackrel{\mathrm{def}}{=}&\sum_{i=1}^n
\biggl(\frac{l V'(x_i)}{\sqrt
{n}}\tilde{G}^i+\frac{l^2V''(x_i)}{2n} \biggr)
\\
&\sim&
\mathcal{N} \biggl(\frac{l^2}{2} \bigl\langle\nu_n,V''
\bigr\rangle, l^2 \bigl\langle\nu _n,
\bigl(V' \bigr)^2 \bigr\rangle-l^2
\bigl(V'(x_1) \bigr)^2/2 \biggr),
\\
Y&\stackrel{\mathrm{def}}{=}&\sum_{i=1}^n
\biggl(V \biggl(x_i+\frac{l}{\sqrt{n}}\tilde{G}^i
\biggr)-V(x_i) \biggr)
\\
&=&X+\frac{l^2}{2n}\sum_{i=1}^n
V''(x_i) \bigl( \bigl(
\tilde{G}^i \bigr)^2-1 \bigr)+\frac
{l^3}{6n^{3/2}}\sum
_{i=1}^n V^{(3)}(
\chi_i) \bigl(\tilde{G}^i \bigr)^3
\end{eqnarray*}
with $\chi_i\in[x_i,x_i+\frac{l}{\sqrt{n}}\tilde{G}^i]$. By
boundedness of $V''$ and $V^{(3)}$ and since
$\E[((\tilde{G}^i)^2-1)((\tilde{G}^j)^2-1)]=0$ as soon as $j\neq i$
and $\E[V^{(3)}(\chi_i)(\tilde G^i)^3((\tilde{G}^j)^2-1)]=0$ as soon
as $j\notin\{1,i\}$,
$\E[(X-Y)^2]\leq\frac{C}{n}$ which implies that $W_1(X,Y)\leq\frac
{C}{\sqrt{n}}$. The density of $X$ is bounded by $(l^2\pi\langle\nu
_n,(V')^2\rangle)^{-1/2}$.
By Lemma~\ref{propggam}, the function $\cG$ takes its values in
$[0,l^2]$. Moreover,
\[
\partial_b {\cG}(a,b)=-\frac{l^2}{2}{\cG}(a,b)+
\frac{l^3}{2\sqrt
{2\pi a}}e^{-(l^2b^2)/(8a)}
\]
which ensures that $\sup_{(a,b)\dvtx \llvert b\rrvert \geq a^{1/4}}\llvert \partial_b {\cG
}(a,b)\rrvert <+\infty$. Lemma~\ref{wass} applied with $\varepsilon=\frac
{\sqrt{\llvert V'(x_1)\rrvert }}{(2n)^{1/4}}$ implies that
\begin{eqnarray*}
&& \Biggl\llvert \E \Biggl[{\cG} \Biggl(\frac{(V'(x_1))^2}{2n},\frac
{2}{l^2}
\sum_{i=1}^n \biggl(V
\biggl(x_i+ \frac{l}{\sqrt{n}}\tilde {G}^i
\biggr)-V(x_i) \biggr) \Biggr) \Biggr]
\\
&&\hspace*{128pt}{}-\E \biggl[{\cG} \biggl(
\frac
{(V'(x_1))^2}{2n},\frac{2X}{l^2} \biggr) \biggr] \Biggr\rrvert
\\
&&\qquad \leq\frac
{C}{\sqrt{n}}+\frac{C}{(n\langle\nu_n,(V')^2\rangle)^{1/4}}+\frac
{C\sqrt{\llvert V'(x_1)\rrvert }}{n^{1/4}\sqrt{\langle\nu_n,(V')^2\rangle}},
\end{eqnarray*}
where $C$ depends neither on $x$ nor on $n$.
One concludes by remarking that, by~(\ref{calesp5}),
\[
\E \biggl[{\cG} \biggl(\frac{(V'(x_1))^2}{2n},\frac{2X}{l^2} \biggr) \biggr]={
\cG} \bigl( \bigl\langle\nu_n, \bigl(V'
\bigr)^2 \bigr\rangle, \bigl\langle\nu _n,V''
\bigr\rangle \bigr).
\]\upqed
\end{pf*}
\end{appendix}



%

\printaddresses
\end{document}